\documentclass[12pt]{amsart}
\usepackage{graphicx,hyperref}
\usepackage{amsthm,amsfonts,latexsym,amssymb,cite}

\newtheorem{proposition}{\bf{Proposition}}[section]

\newtheorem{remark}{\sc{Remark} }[section]
\newtheorem{defn}{\sc{Definition} }[section]

\def\build#1_#2{\mathrel{\mathop{\kern 0pt#1}\limits_{#2}}}
  
\begin{document}

\title[Explicit estimates in mixed elliptic problems]{Explicit estimates for  solutions of  mixed elliptic problems}
\author{Luisa Consiglieri}
\address{Luisa Consiglieri, Independent Researcher Professor, European Union}
\urladdr{\href{http://sites.google.com/site/luisaconsiglieri}{http://sites.google.com/site/luisaconsiglieri}}

\begin{abstract}
We deal with the existence of quantitative estimates for
 solutions of mixed problems to an elliptic 
second order equation in divergence form with discontinuous coefficient.
Our concern is to estimate the solutions with explicit 
constants, for domains in 
$\mathbb{R}^n$ ($n\geq 2$) of class $C^{0,1}$.
 The existence of  $L^\infty$ and
 $W^{1,q}$-estimates is assured for $q=2$ and any $q<n/(n-1)$
(depending on the data), 
whenever the coefficient is only measurable and bounded.
The proof method of the quantitative
 $L^\infty$-estimates is based on the DeGiorgi technique
developed by Stampacchia.
By using the potential theory,
 we derive $W^{1,p}$-estimates for different ranges of the exponent
$p$ depending on that the coefficient is either
Dini-continuous or  only measurable and bounded.
In this process, we establish new 
 existences of Green functions on such domains.
The last but not least concern is to unify (whenever possible)
the proofs of the estimates to the extreme Dirichlet and Neumann cases
of the mixed problem.
\end{abstract}



\maketitle
\section{Introduction}

The knowledge of the 
data makes all the difference on the real world
 applications of boundary value problems. 
Quantitative estimates are of extremely importance in any other area of science
such as engineering, biology, geology, even physics,
to mention a few.
In the existence theory to the nonlinear elliptic equations,
 fixed point arguments play a crucial role.
The solution may exist such that belongs to a bounded set of a
functional space,  where the
boundedness constant is frequently given in an abstract way. 
Their derivation is so complicated that it is difficult to express them,
or they include unknown ones that are achieved by a contradiction proof,
as for instance the Poincar\'e constant for nonconvex domains.
The majority
of works consider the same symbol for any constant that varies from line to line
along the whole paper (also known as universal constant).
 In conclusion, the final constant of the boundedness
appears completely unknown from the physical point of view.
In presence of this, our first concern is to explicit the dependence on the data
of the boundedness constant. To this end,
first (Section \ref{hilb}) we solve in $H^{1}$ the Dirichlet, mixed
 and Neumann problems
to an elliptic second order equation in divergence form with discontinuous coefficient, and simultaneously
we establish the quantitative estimates with explicit constants.
 Besides in Section \ref{l1data}
we derive $W^{1,q}$ ($q<n/(n-1)$)
estimative constants involving $L^1$ and measure data,
 via the technique of
 solutions obtained by limit approximation (SOLA)  (cf. \cite{bg,aa,lap,p}).

Dirichlet, Neumann, and mixed
 problems with respect to uniformly elliptic equation in divergence form is 
widely investigated
 in the literature (see \cite{adn,bv72,dauge,gia93,gt,grisv,lu,Ni,stamp63-64}
 and the references therein)
when the leading coefficient is a  function  on the spatial variable,
 and the boundary values are  given by  assigned Lebesgue functions.
  Meanwhile, many results on the regularity for elliptic PDE
are appearing \cite{agranovich,bw2004,bw2005,dif,dong,
elschk,geng,gro89,groreh,haller,lv,mey,ott,rag99,sav} 
(see Section \ref{reg} for details).
Notwithstanding their estimates seem to be inadequate for physical and
technological applications. For this reason,
the explicit description of the estimative constants needs to carry out.
Since the smoothness of the solution 
is invalidated by the nonsmoothness of the coefficient and the domain,
Section \ref{linfty} is devoted to the direct derivation of 
global and local $L^\infty$-estimates.

It is known that the information 'The gradient of a quantity
belongs to a $L^p$ space with $p$ larger than the space dimension'
is extremely useful for the analysis
of boundary value problems to nonlinear elliptic equations in divergence form
with leading coefficient 
 $a(x,T)=a(x,T(x))\in L^\infty(\Omega)$, where $T$ is 
  a known  function, usually the temperature function, such as
 the electrical conductivity in
the thermoelectric   \cite{aduf,zamm} and thermoelectrochemical \cite{epjp}
problems. It is also known that one cannot expect in general
that the integrability exponent for the gradient of the solution
of an elliptic equation exceeds a prescribed number
$p>2$, as long as arbitrary elliptic $L^\infty$-coefficients
are admissable  \cite{elschk}.
Having this in mind, in Section \ref{reg} we derive
$W^{1,p}$-estimates  of weak solutions,
which verify the representation formula,
of the 
Dirichlet, Neumann, and mixed problems to
an elliptic second order equation in divergence form.
The proof is based on the existence of 
 Green kernels, which are described in Section \ref{secg},
whenever the coefficients are whether
continuous or only measurable and bounded
(inspired in some techniques from \cite{gw,lsw,kp}).

\section{Statement of the problem}

Let $\Omega$ be a  domain (that is, connected open set) in 
$\mathbb{R}^n$ ($n\geq 2$) of class $C^{0,1}$, and bounded.
Its boundary $\partial\Omega$ is constituted by two disjoint
 open $(n-1)$-dimensional sets,
$\Gamma_D$ and $\Gamma$,
such that $\partial\Omega=\bar\Gamma_D\cup \bar\Gamma$. 
The Dirichlet situation $\Gamma_D=\partial\Omega$
(or equivalently $\Gamma=\emptyset$), and the Neumann
 situation $\Gamma=\partial\Omega$
(or equivalently $\Gamma_D=\emptyset$) are available.

Let us consider the following boundary value problem, in the sense of distributions,
\begin{eqnarray}
-\nabla\cdot(a\nabla u)=f-\nabla\cdot{\bf f}&\mbox{ in }&\Omega;\label{omega}\\
(a\nabla u-{\bf f})\cdot{\bf n}=h&\mbox{ on }&\Gamma;\\
u=g&\mbox{ on }&\Gamma_D, \label{gama}
\end{eqnarray}
where $\bf n$ is the unit outward normal to the boundary $\partial\Omega$.

Set for any $q\geq 1$
\[V_q=\left\{
\begin{array}{ll}
W^{1,q}_{\Gamma_D}(\Omega)=\{v\in W^{1,q}(\Omega):\ v=0\mbox{ on }\Gamma_D\}
&\mbox{ if }|\Gamma_D|>0\\
\left.\begin{array}{l}
V_q(\partial\Omega)=
\{v\in W^{1,q}(\Omega):\ \int_{\partial\Omega} v(x) \mathrm{ds}=0\}\\
V_q(\Omega)=
\{v\in W^{1,q}(\Omega):\ \int_{\Omega} v(x) \mathrm{dx}=0\}
\end{array}\right\}
&\mbox{ otherwise}
\end{array}\right.
\]
the Banach space endowed with the seminorm of $W^{1,q}(\Omega)$,
 taking the Poincar\'e inequalities (\ref{poin1})-(\ref{poin2}) into account,
 since 
any bounded Lipschitz domain has the cone property.
Here $|\cdot|$ stands for the $(n-1)$-Lebesgue measure.
Also $|A|$ stands for the Lebesgue measure of a set $A$ of $\mathbb{R}^n$.
The significance of $|\cdot|$ depends on the kind of the set.

 Defining the $W^{1,q}$-norm by
 \[ 
\| v\|_{1,q,\Omega}:={1\over C_*+1}(\|v\|_{q,\Omega}+
\|\nabla  v\|_{q,\Omega}),
\]
with $C_*$ being anyone of the  Poincar\'e constants
\begin{eqnarray}
\|v- -\hspace*{-0.4cm}\int_{\Sigma }
v\mathrm{ds}\|_{q,\Omega}\leq C_*
\|\nabla  v\|_{q,\Omega},&&\forall v\in W^{1,q}(\Omega);\label{poin1}\\
\|v- -\hspace*{-0.4cm}\int_{\Omega }
v\mathrm{dx}\|_{q,\Omega}\leq C_*
\|\nabla  v\|_{q,\Omega},&&\forall v\in W^{1,q}(\Omega),\label{poin2}
\end{eqnarray}
where $\Sigma\subset\partial \Omega$, and
$ -\hspace*{-0.35cm}\int_{A}$ means the integral average over the set $A$
of positive measure,
 the Sobolev and trace inequalities read
\begin{eqnarray}\label{sobs}
\| v\|_{q^*,\Omega}\leq S_q
\|\nabla  v\|_{q,\Omega};&&\\
\label{sobk}
\| v\|_{q_*,\partial\Omega}\leq K_q
\|\nabla  v\|_{q,\Omega},&&\quad\forall v\in V_q.
\end{eqnarray}
Hence further we call (\ref{sobs}) the Sobolev inequality,
 and for the general situation  the $W^{1,q}$-Sobolev inequality.
Analogously, the trace inequality may be stated.
For  $1\leq q<n$,  $q^*=qn/(n-q)$  and
 $q_*=q(n-1)/(n-q)$  are the critical Sobolev and trace exponents such that correspond, respectively, to
$ {W}^{1,q}(\Omega)
 \hookrightarrow { L}^{q^*}(\Omega)$ and
  $ W^{1,q}(\Omega)\hookrightarrow L^{q_*}(\partial\Omega)$.
For $1<q<n$,
 the best constants of the Sobolev and trace inequalities are, respectively, (for smooth functions that decay at infinity, see \cite{tale} and
\cite{bond})
\begin{eqnarray*}
S_q&=&\pi^{-1/2}n^{-1/q}\left({q-1\over n-q}\right)^{1-1/q}\left[
{\Gamma(1+n/2)\Gamma(n)\over \Gamma (n/q)\Gamma(1+n-n/q)}\right]^{1/n};
\\ 
K_q&=&\pi^{1-q\over 2}\left({q-1\over n-q}\right)^{q-1}\left[
{\Gamma\left({q(n-1)\over 2(q-1)}\right)\Big/ \Gamma \left({
n-1\over 2(q-1)}\right)}\right]^{q-1\over n-1}.
\end{eqnarray*}
We observe that $q^*>1$ is arbitrary if $q=n$. Here $\Gamma$ stands for the
gamma function.
Set by $\omega_n$ the volume of the  unit ball $B_1(0)$ of
$\mathbb{R}^n$, that is, $\omega_n=\pi^{n/2}/\Gamma(n/2+1)$
and $\Gamma(n/2+1)=(n/2)!$ if $n$ is even, and 
$\Gamma(n/2+1)=\pi^{1/2}2^{-(n+1)/2}n(n-2)(n-4)\cdots 1$ if $n$ is odd.
Moreover, the relationship $\sigma_{n-1}=n\omega_n$ holds true,
where $\sigma_{n-1}=2\pi^{n/2}/\Gamma(n/2)$
 denotes the area of the  unit sphere $\partial B_1(0)$.

For $n>1$, from the fundamental theorem of calculus applied to
each of the $n$ variables separately, it follows that
\begin{equation}\label{sobn}
\|u\|_{n/(n-1),\Omega}\leq n^{-1/ 2}\|\nabla u\|_{1,\Omega}.
\end{equation}
We emphasize that the above explicit constant is not sharp,
since there exists the limit constant
$S_1=\pi^{-1/2}n^{-1}[\Gamma(1+n/2)]^{1/n}$ \cite{tale}.

\begin{defn}\label{def1}
We say that $u$ is weak solution to (\ref{omega})-(\ref{gama}), 
if it 
verifies
$u=g $ a.e. on $\Gamma_D$, and
\begin{equation}\label{pbumax}
\int_{\Omega}     a\nabla u\cdot
\nabla v \mathrm{dx}
=\int_{\Omega}{\bf f}\cdot\nabla v \mathrm{dx}
+\int_{\Omega}fv \mathrm{dx}+\int_{\Gamma}hv \mathrm{ds}, \quad\forall v\in V_2,
\end{equation}
where  $g\in L^2(\Gamma_D)$,
 ${\bf f}\in {\bf L}^{2}(\Omega)$,  $f\in L^{t}(\Omega)$, with
 $t=(2^*)'$, i.e. $t=2n/(n+2)$
if $n>2$ and any $t>1$ if $n=2$, $h\in L^{s}(\Gamma)$, with $s=2(n-1)/n$ 
if $n>2$ and any $s>1$ if $n=2$,
and 
 $a\in L^\infty(\Omega)$ satisfies $0< a_\#\leq a\leq a^\#$ a.e. in $\Omega$.
\end{defn}

Since $\Omega$ is bounded, we have that $\Omega\subset B_{\partial(\Omega)}(x)$,
where $\delta(\Omega):={\rm diam}(\Omega)$, for every $x\in\Omega$.
We emphasize that the existence of equivalence between the strong
(\ref{omega})-(\ref{gama}) and weak (\ref{pbumax}) formulations
is only available under sufficiently data.
For instance, the Green formula may be applied if
$a\nabla u \in {\bf L}^2(\Omega)$ and $\nabla\cdot(a\nabla u)\in L^2(\Omega)$.

\section{Some $W^{1,q}$-constants ($q\leq 2$)}
\label{solve}

The presented results in this Section are valid
whether $a$ is  a matrix or a function such that obeys the measurable and boundedness properties.
We emphasize that in the matrix situation
$a\nabla u\cdot\nabla v=a_{ij}\partial_i u\partial_j v$, under the Einstein summation convention.
Here we restrict to the function situation for the sake of simplicity.

\subsection{$H^{1}$-solvability}
\label{hilb}

We recall the existence result in the Hilbert space $H^1$ in order to
express its explicit constants in the following propositions,
namely Propositions \ref{exist} and \ref{neumann} corresponding to
the mixed and the Neumann problems, respectively.
\begin{proposition}\label{exist}
If $|\Gamma_D|>0$, then there exists $u \in H^{1}(\Omega)$ 
being a weak solution to (\ref{omega})-(\ref{gama}).
If $g=0$, then $u$ is unique.
Letting  $\widetilde g\in H^1(\Omega)$ as an  extension of
$g\in L^2(\Gamma_D)$, i.e. it is
such that $\widetilde g=g$ a.e. on $\Gamma_D$,
 the following estimate holds
\begin{eqnarray}\label{dircota}
\|\nabla u\|_{2,\Omega}&\leq& (a^\#/a_\#+1)\|
\nabla\widetilde g\|_{2,\Omega}+\\
&+&{1\over a_\#}\left(
\|{\bf f}\|_{2,\Omega}+C_{n}(\|f\|_{t,\Omega},
 \|h\|_{s,\Gamma})
\right),\nonumber
\end{eqnarray}
where $C_n(A,B)=S_2A+K_2B$
if $n>2$,  $C_2(A,B)= |\Omega |^{1/t'} S_{2t/(3t-2)} A+ |\Omega |^{1/(2s')}
K_{2s/(2s-1)}B$ if $t<2$, and 
 $C_2(A,B)= |\Omega |^{1/t'} A/\sqrt{2}+ |\Omega |^{1/(2s')}\\
K_{2s/(2s-1)}B$ if $t\geq 2$. In particular, 
 $u-\widetilde g\in H^1_{\Gamma_D}(\Omega)$ is unique.
\end{proposition}
\begin{proof}
For $g\in L^2(\Gamma_D)$ there exists an extension $\widetilde g\in H^1(\Omega)$
such that $\widetilde g=g$ a.e. on $\Gamma_D$.
The existence and uniqueness of a weak solution
$w\in H^1_{\Gamma_D}(\Omega)$ is well-known via the Lax-Milgram Lemma,
to the variational problem
\begin{equation}\label{pbw}
\int_{\Omega}     a\nabla w\cdot
\nabla v \mathrm{dx}=\int_{\Omega}({\bf f}-
a\nabla\widetilde g)\cdot\nabla v \mathrm{dx}
+\int_{\Omega}fv \mathrm{dx}+\int_{\Gamma}hv \mathrm{ds}, 
\end{equation}
for all $v\in  H^1_{\Gamma_D}(\Omega)$. 
 Therefore, the required solution is given by $u=w+\widetilde g$.

If $g=0$, $\widetilde g=0$ and then $u\equiv w$.

 Taking $v=w\in H^1_{\Gamma_D}(\Omega)$ as a test function in (\ref{pbw}),
 applying the H\"older inequality, and using the lower and upper bounds of $a$,
 we obtain
 \[
 a_\#\|\nabla w\|_{2,\Omega}^2\leq \left(\|{\bf f}\|_{2,\Omega}
+a^\#\|\nabla\widetilde g\|_{2,\Omega}\right)\|\nabla w\|_{2,\Omega}
 +\|f\|_{t,\Omega}\|w\|_{t',\Omega}+
\|h\|_{s,\Gamma}\|w\|_{s',\Gamma}.
\]

For $n>2$,  this inequality reads
 \[
 a_\#\|\nabla w\|_{2,\Omega}\leq \|{\bf f}\|_{2,\Omega}
+a^\#\|\nabla\widetilde g\|_{2,\Omega} +S_2\|f\|_{2n/(2+n),\Omega}+
K_2\|h\|_{2(n-1)/n,\Gamma},
\] 
implying (\ref{dircota}).

Consider the case of dimension $n=2$.
For $t,s>1$, using the H\"older inequality in (\ref{sobn})
if $t'\leq 2$, in (\ref{sobs}) if $t'>2$, and in (\ref{sobk}) for any $s>1$,
we have
\begin{eqnarray*}
\|w\|_{t',\Omega}\leq|\Omega |^{1/t'-1/2}\|w\|_{2,\Omega}\leq
 {1\over \sqrt{2}}|\Omega |^{1/t'-1/2}\|\nabla w\|_{1,\Omega} 
\leq {1\over \sqrt{2}}|\Omega |^{1/t'}\|\nabla w\|_{2,\Omega},
\quad (t\geq 2);\\
\|w\|_{t',\Omega}\leq S_{2t/(3t-2)}\|\nabla w\|_{2t/(3t-2),\Omega} \leq
  S_{2t/(3t-2)}|\Omega |^{1/t'}\|\nabla w\|_{2,\Omega} ,\quad
(t<2);\\
\|w\|_{s',\Gamma}\leq K_{2s/(2s-1)}\|\nabla w\|_{2s/(2s-1),\Omega} \leq
  K_{2s/(2s-1)}|\Omega |^{1/(2s')}\|\nabla w\|_{2,\Omega} .\end{eqnarray*}
This concludes the proof of Proposition \ref{exist}.
\end{proof}

\begin{proposition}[Neumann]\label{neumann}
If $|\Gamma_D|=0$,  then there exists a unique $u \in V_2$
being a weak solution to (\ref{omega})-(\ref{gama}). Moreover,
 the following estimate holds
\begin{equation}\label{neumcota}
\|\nabla u\|_{2,\Omega}\leq{1\over a_\#}\left(
\|{\bf f}\|_{2,\Omega}+C_{n}(\|f\|_{t,\Omega},
 \|h\|_{s,\Gamma})
\right),
\end{equation}
where $C_n(A,B)$ is given as in Proposition \ref{exist}.
\end{proposition}
\begin{proof}
The existence and uniqueness of a weak solution
$u\in V_2$ is consequence of the Lax-Milgram Lemma
(see Remark \ref{rneum}).
The estimate (\ref{neumcota}) follows the same argument used to prove
(\ref{dircota}).
\end{proof}

\begin{remark}\label{rneum}
The meaning of the Neumann solution $u\in V_2$
in Proposition \ref{neumann} should be understood as  $u\in V_2(\partial\Omega)$
solving (\ref{pbumax}) for all   $v\in V_2(\partial\Omega)$, or
 $u\in V_2(\Omega)$
solving (\ref{pbumax}) for all   $v\in V_2(\Omega)$.
\end{remark}

\subsection{$W^{1,q}$-solvability ($q\leq n/(n-1)$)}
\label{l1data}

The existence of a  solution is recalled in the following
proposition in
accordance to $L^1$-theory, that is 
 via solutions obtained by limit approximation (SOLA)  (cf. \cite{bg,aa,lap,p}),
 in  order to determine the explicit constants.
\begin{proposition}\label{W1q}
Let  $g=0$  on $\Gamma_D$ (possibly empty),
${\bf f}\in{\bf L}^2(\Omega)$, $f\in L^1(\Omega)$,
$h\in L^1(\Gamma)$, and $a\in L^\infty(\Omega)$
satisfy $0< a_\#\leq a\leq a^\#$ a.e. in $\Omega$.
For any $1\leq q < n/(n-1)$ there exists $u\in
  V_{q}$ solving (\ref{pbumax}) for every $v\in V_{q'}$.
Moreover, we have the following estimate
\begin{eqnarray}\label{cota1qv}
\| \nabla u \|_{q,\Omega}& \leq&
C_1(\Omega,n,q) \left(  {\|{\bf f}\|_{2,\Omega}\over 
a_\#}+ \sqrt{\varkappa (\|f\|_{1,\Omega}+\|h\|_{1,\Gamma})\over 
a_\#}\right)\\ &&
+C_2 \left( n,q, {\|{\bf f}\|_{2,\Omega}\over 
a_\#}+\sqrt{\varkappa(\|f\|_{1,\Omega}+\|h\|_{1,\Gamma})\over 
a_\#}
\right),\nonumber
\end{eqnarray}
with $\varkappa=2$  if $|\Gamma_D|>0$, $\varkappa=4$  if $|\Gamma_D|=0$, and
\begin{eqnarray*}
C_1(\Omega,n,q)= |\Omega|^{1/q-1/2}\left\{
\begin{array}{ll}
{(n-q)^{3/2}\over q(n-2)(n+q-nq)^{1/2}}
2^{\frac {1}{q}+\frac {3n-q(n+1)}{2(n-q)}}&\mbox{ if }n>2\\
(2-q)^{-1/2}2^{\frac{(6-q)q-2}{2q}}&\mbox{ if }n=2
\end{array}\right.\\
C_2(n,q,A)=\left\{
\begin{array}{ll}
A^{\frac {2(n-q)}{q(n-2)}}\left({n-q\over n+q-nq}\right)^{\frac {n-q}{q(n-2)}}
2^{\frac {(2-q)(nq-n+q)}{q^2(n-2)}}
S_q^{\frac {n(2-q)}{q(n-2)}}  &\mbox{ if }n>2\\
  A^{\frac {2}{q-1}}
|\Omega|^{1/q-1/2} {\left[(q-1)(3-q)^{\frac {3-q}{q-1}}
\right]^{\frac{2-q}{2q}}\over (2-q)^{1/(q-1)}} 2^\ell
S_q^{\frac {3-q}{q-1}}&\mbox{ if }n=2
\end{array}\right.
\end{eqnarray*}
where $\ell\in ]2,+\infty[$ is explicitly given in (\ref{ell}).
\end{proposition}
\begin{proof}
For each $m\in \mathbb N$, take
\[
f_m={m f\over m+|f|}\in L^\infty(\Omega),\quad
h_m={m h\over m+|h|}\in L^\infty(\Gamma).
\]
Applying Propositions \ref{exist} and \ref{neumann},
there exists a unique solution $u_m\in V_2
$ to the following variational problem
\begin{equation}\label{pbum}
\int_\Omega a\nabla u_m\cdot\nabla v\mathrm{dx}
=\int_\Omega
{\bf f}\cdot\nabla v\mathrm{dx}+
\int_\Omega f_mv\mathrm{dx}+\int_\Gamma
h_mv\mathrm{ds},\quad\forall v\in V_2.
\end{equation}
In particular, (\ref{pbum}) holds for all $v\in V_{q'}$ ($q'>n$).

In order to pass to the limit (\ref{pbum}) on $m$ $(m\rightarrow\infty)$
let us establish the estimate (\ref{cota1qv}) for
 $ \nabla u_m$.
 
{\sc Case} $|\Gamma_D|>0$.
From $L^1$-data theory (see, for instance, \cite{p}),
let us choose 
\[ 
v={{\rm sign}(u_m)}[1-1/(1+|u_m|)^s ]\in
W^{1,2}_{\Gamma_D}(\Omega)\cap L^\infty(\Omega),\quad
\mbox{for }s  >0,
\]
as a test function in (\ref{pbum}). Hence it follows that
\[
a_\#\int_\Omega{s |\nabla u_m|^2\over
(1+| u_m|)^{s +1}}\mathrm{dx}
\leq s\|{\bf f}\|_{2,\Omega}\|{\nabla u_m\over
(1+| u_m|)^{s +1\over 2}}\|_{2,\Omega}+\|f\|_{1,\Omega}+\|h\|_{1,\Gamma},
\]
and consequently
\[
\int_\Omega{ |\nabla u_m|^2\over
(1+| u_m|)^{s +1}}\mathrm{dx}
\leq {1\over (a_\#)^2}\|{\bf f}\|_{2,\Omega}^2+
{2\over sa_\#}(\|f\|_{1,\Omega}+\|h\|_{1,\Gamma}).
\]

By the H\"older inequality with exponents $2/q$ and $2/(2-q)>1$, we have
\begin{equation}\label{hold}
 \int_{\Omega} |\nabla u_m |^q \mathrm{dx} \leq
 \left( 
\int_{\Omega} \frac{ |\nabla u_m|^2 } { (1+ |u_m
 |)^{s+1}}  \mathrm{dx}  \right)^{\frac q2}
\left( \int_{\Omega} (1+ |u_m |)^{\frac {(s+1) q}{2-q}}  \mathrm{dx} 
 \right)^{\frac {2-q}2} .
\end{equation}
Set
\[ 
M(s):=
{\|{\bf f}\|_{2,\Omega}\over 
a_\#}+\left({2(\|f\|_{1,\Omega}+\|h\|_{1,\Gamma})\over 
a_\#s}\right)^{1/2}.
\]

Let us choose  $s>0$ such that $(s+1)q/(2-q)=q^*=nq/(n-q)$
which is possible since $1\leq q < n/(n-1)$, that is $s=
(n+q-nq)/(n-q)$.
Then, gathering the above two inequalities, and inserting
(\ref{sobs}) for
$u_m\in V_2\hookrightarrow V_q$ with $(q\leq 2)$, we deduce
\begin{eqnarray*}
\| \nabla u_m \|_{q,\Omega} \leq 
M(s)2^{\frac {2-q}{2q}(\frac {nq}{n-q}-1)}
\left(|\Omega|^{1/q-1/2}
+(S_q
\| \nabla u_m \|_{q,\Omega})^{\frac {n(2-q)}{2(n-q)}}  \right)\\
\leq M({n+q-nq \over n-q})
2^{\frac {2-q}{2q}(\frac {nq}{n-q}-1)} |\Omega|^{1/q-1/2}
+{\frac {n(2-q)}{2(n-q)}}
\| \nabla u_m \|_{q,\Omega}+ \\
+
{\frac {q(n-2)}{2(n-q)}} 
[M({n+q-nq \over n-q})]^{\frac {2(n-q)}{q(n-2)}}
2^{\frac {(2-q)(nq-n+q)}{q^2(n-2)}}
S_q^{\frac {n(2-q)}{q(n-2)}}  ,
\end{eqnarray*}
using the Young inequality $AB\leq \epsilon A^a/a+B^b/(b\epsilon ^{b/a})$,
for $A,B\geq 0$, $\epsilon>0$, and $a,b>1$ such that $1/a+1/b=1$,
 with $\epsilon=1$, and $a=2(n-q)/[n(2-q)]$ if $n>2$.

For $n=2$,  $s>0$ is chosen such that $(s+1)q/(2-q)<q^*=2q/(2-q)$
which is possible since $1\leq q < 2$, that is $s<1$.
Using the above Young inequality with $a=2/(s+1)$, we find   
 \begin{eqnarray*}
(1+|u_m|) ^{(s+1)q/(2-q)}\leq
2^{\frac{(s+1)q}{2-q}-1}(1+|u_m| ^{(s+1)q/(2-q)})\leq\\
\leq 2^{\frac{(s+2)q-2}{2-q}}+{1-s\over 2}
\left({s+1\over 2\epsilon } \right)^{(s+1)/(1-s)}
2^{\frac{2[(s+2)q-2]}{(2-q)(1-s)}}
+\epsilon|u_m|^{q^*}.
\end{eqnarray*}
Let us choose, for instance,  $s=2-q<1$, and
$\epsilon =[2S_qM(s)]^{ -2q/(2-q)} $.
Then, we obtain
\begin{eqnarray*}
\| \nabla u_m \|_{q,\Omega} \leq  M(s)\left(
\epsilon^{(2-q)/(2q)}S_q
\| \nabla u_m \|_{q,\Omega}+\right.\\
 \left. +|\Omega|^{1/q-1/2}
\left[2^{\frac{(s+2)q-2}{2-q}}+ {1-s\over 2}
\left({s+1\over 2\epsilon } \right)^{(s+1)/(1-s)}
2^{\frac{2[(s+2)q-2]}{(2-q)(1-s)}}
\right]^{\frac{2-q}{2q}}
 \right)\\
\leq 
{\frac {1}{2}}  \| \nabla u_m \|_{q,\Omega}+ 
M (2-q) |\Omega|^{1/q-1/2}\times
\left[2^{\frac{(4-q)q-2}{2q}}+\right.\\
\left.+{1\over 2}\left[(q-1)(3-q)^{\frac {3-q}{q-1}}
\right]^{\frac{2-q}{2q}} 
2^\ell [M(2-q)]^{\frac {3-q}{q-1}} 
S_q^{\frac {3-q}{q-1}} \right] ,
\end{eqnarray*}
where $\ell$ is given by
\begin{equation}\label{ell}
\ell=\frac {2[(s+2)q-2]-(2-3q)(s+1)}{2q(1-s)}
=\frac {-5q^2+19q-10}{2q(q-1)}\rightarrow 2 ^+,
\end{equation}
as $q\rightarrow 2^-$. Hence, we find (\ref{cota1qv}) with $\varkappa=2$.

{\sc Case} $|\Gamma_D|=0$.
We choose, for $s >0$,
\begin{eqnarray*}
v=-{{\rm sign}(u_m)\over (1+|u_m|)^s }+-\hspace*{-0.4cm}\int_{\partial\Omega} 
{{\rm sign}(u_m)\over (1+|u_m|)^s }\mathrm{ds}\in
V_2(\partial\Omega);\\
v=-{{\rm sign}(u_m)\over (1+|u_m|)^s }+-\hspace*{-0.4cm}\int_{\Omega} 
{{\rm sign}(u_m)\over (1+|u_m|)^s }\mathrm{dx}\in
V_2(\Omega),
\end{eqnarray*}
as a test function in (\ref{pbum}). 
Since $|v|\leq 2$ a.e. in $\Omega$, it follows that
\[
\int_\Omega{ |\nabla u_m|^2\over (1+| u_m|)^{s +1}}\mathrm{dx}
\leq  {1\over (a_\#)^2}\|{\bf f}\|_{2,\Omega}^2+
{4\over sa_\#}(\|f\|_{1,\Omega}+\|h\|_{1,\Gamma}).
\]
Then, we argue as in the above case, concluding (\ref{cota1qv})
 with $\varkappa=4$.

For both cases, we can extract a subsequence of $u_m,$ still denoted by
$u_m,$ such that it
weakly converges to $u$ in $W^{1,q}(\Omega)$, where $u\in V_q$
 solves the
limit problem (\ref{pbumax}) for all $v\in V_{q'}$.
\end{proof}

\begin{remark}
In terms of Proposition \ref{W1q},
 the terms on the right hand side of  (\ref{pbumax})
 have sense, since $v\in
  W^{1,q'}(\Omega)\hookrightarrow C(\bar\Omega)$ for $q'>n,$
  that is, $q<n/(n-1)$.
\end{remark}
 
\begin{remark}
The existence of a  solution, which is given at  Proposition \ref{W1q},
is in fact unique  for
the class of SOLA solutions (cf. \cite{bg,aa,lap}).
By the uniqueness of solution in the Hilbert space,
this unique SOLA solution is the weak solution of $V_2$,
if the data belong to the convenient $L^2$ Hilbert spaces.
\end{remark}

Finally, we state the following  version of Proposition \ref{W1q},
which will be required in Section  \ref{secg}, with datum belonging to
 the space of all signed
measures with finite total variation  $\mathcal{M}(\Omega)
=\left( {C}_0 (\Omega)\right)'$. 
\begin{proposition}\label{dirac1}
Let  $g=0$  on $\Gamma_D$ (possibly empty),  $a\in L^\infty(\Omega)$
satisfy $0< a_\#\leq a\leq a^\#$ a.e. in $\Omega$, and
for each $x\in\Omega$,
$\delta_x\in \mathcal{M}(\Omega)$
 be the Dirac delta function.
For any $1\leq q < n/(n-1)$ there exists $u\in  V_{q}$ solving 
  \[
\int_{\Omega}     a\nabla u\cdot
\nabla v\mathrm{dx}=
\langle \delta_x, v\rangle_{\mathcal{M} (\Omega)
\times  {C}_0 (\Omega)}, 
 \quad \forall v\in C_0 (\Omega)\cap V_{1},
\]
for every $i=1,\cdots ,n$.
Moreover, we have the following estimate
\begin{equation}
\| \nabla u \|_{q,\Omega} \leq
C_1(\Omega,n,q)\sqrt{\varkappa  /a_\#}
+C_2 ( n,q, \sqrt{\varkappa /a_\#}),\label{cotad}
\end{equation}
where the constants $C_1(\Omega,n,q)$, $C_2 ( n,q, A)$, and $\varkappa$
 are determined in  Proposition \ref{W1q}.
\end{proposition}
\begin{proof} 
Since  the Dirac delta function $\delta_x\in \mathcal{M} (\Omega)$  can be
approximated by a sequence $\{f_m\}_{m\in\mathbb N}
\subset L^\infty(\Omega)$ such that
\[
\|f_m\|_{1,\Omega}= 1,\quad \mbox{and}\quad
\lim_{
m\rightarrow \infty}\int_\Omega f_m\varphi\mathrm{dx}=
\langle \delta_x,\varphi\rangle_{\mathcal{M} (\Omega)
\times  {C}_0 (\Omega)},   \ \forall \varphi\in 
{ C}_0 (\Omega),
\]
the identity (\ref{pbum}) holds,   with $f$ replaced by $f_m$, 
${\bf f}=0$ in $\Omega$, and $h=0$ on $\Gamma$,
 for all $v\in V_2$ and in particular for all 
$v\in C_0 (\Omega)\cap V_{1}$.
 Then, we may proceed by using
the  argument already used in the  proof of  Proposition \ref{W1q}, 
 with $\|f\|_{1,\Omega}=1$, and $\|{\bf f}\|_{1,\Omega}=\|h\|_{1,\Gamma}=0$,
 to conclude (\ref{cotad}).
\end{proof}
 
\section{$L^\infty$-constants}
\label{linfty}

 In this Section, we establish some maximum principles, by recourse to
 the De Giorgi technique \cite{stamp63-64},
via the analysis of the decay of the level sets of the solution. We begin by
deriving the explicit estimates in the mixed case $|\Gamma_D|>0$.
\begin{proposition}\label{max}
Let $p>n\geq 2$,  $|\Gamma_D|>0$, and $u \in H^{1}(\Omega)$ 
be any weak solution to (\ref{omega})-(\ref{gama})
in accordance with Definition \ref{def1}.
If  $g\in L^\infty(\Gamma_D)$,
 ${\bf f}\in {\bf L}^{p}(\Omega)$,  $f\in L^{np/(p+n)}(\Omega)$, and  $h\in L^{(n-1)p/n}(\Gamma)$,
 then we have
\begin{eqnarray}\label{supess}
{\rm  ess } \sup_{\Omega}|u|\leq {\rm  ess } \sup_{\Gamma_D}|g|+\\
+{C_n\over a_\#}|\Omega|^{1/n-1/p}\left(
\|{\bf f}\|_{p,\Omega}+S_{p'}\|f\|_{np/(p+n),\Omega}+K_{p'}
 \|h\|_{(n-1)p/n,\Gamma}
\right),\nonumber
\end{eqnarray}
where $C_n=2^{n(p-2)/[2(p-n)]}S_2$
if $n>2$, and $C_2=2^{(3p-2)/[2(p-2)]}$.
\end{proposition}
\begin{proof} 
Let $k\geq k_0={\rm ess} \sup\{|g(x)|:\ x\in \Gamma_D\}$.
Choosing $v={\rm sign}(u)(|u|-k)^+={\rm sign}(u)\max\{|u|-k,0\}\in
 H^1_{\Gamma_D}(\Omega)$ as a 
test function in (\ref{pbumax}), then $\nabla v=\nabla u\in {\bf L}^2(A(k))$,
and  we deduce
\begin{eqnarray}\label{varak}
a_\#\int_{A(k)}|\nabla u|^2\mathrm{dx}
\leq\|{\bf f}\|_{2,A(k)}\|\nabla u\|_{2,A(k)}+ \\
+\|f\|_{{np\over p+n},\Omega}\|(|u|-k)^+\|_{{np\over
np-n-p},\Omega}+
\|h\|_{(n-1)p/n, \Gamma}\|(|u|-k)^+\|_{{(n-1)p\over np-n-p},\Gamma},\nonumber
\end{eqnarray}
where $A(k)=\{x\in \Omega:\ |u(x)|>k\}$.
Using the H\"older inequality, it follows that
\begin{eqnarray*}
\|{\bf f}\|_{2,A(k)}
\leq \|{\bf f}\|_{p,\Omega}|A(k)|^{1/2-1/p}\qquad (p>2).
\end{eqnarray*}

Making use of (\ref{sobs})-(\ref{sobk}) and
$(|u|-k)^+\in
W^{1,q}_{\Gamma_D}(\Omega)$ with $q=p'<n$, and  the H\"older inequality, we get
\begin{eqnarray*}
\|(|u|-k)^+\|_{np/(n(p-1)-p),\Omega}
&\leq& S_{p'}\|\nabla u\|_{2,A(k)}|A(k)|^{1/p'-1/2};\\
\|(|u|-k)^+\|_{(n-1)p/(n(p-1)-p),\Gamma}
&\leq& K_{p'}\|\nabla u\|_{2, A(k)}|A(k)|^{1/p'-1/2}, 
\end{eqnarray*}
if provided by $p'<2\leq n$.
Inserting last three inequalities into (\ref{varak}) we obtain
\begin{equation}\label{nuak}
\|\nabla u\|_{2,A(k)}\leq C_{n,p}|A(k)|^{1/2-1/p} ,
\end{equation}
where the positive constant $C_{n,p}$ is
\[C_{n,p}=\left(
\|{\bf f}\|_{p,\Omega}+S_{p'}\|f\|_{np/(p+n),\Omega}+K_{p'}
 \|h\|_{(n-1)p/n,\Gamma}
\right)/a_\#,\quad\forall n\geq 2.
\]
Taking into account that $A(h) \subset A(k)$ when $ h > k >k_0$, we find
\begin{equation}\label{aalfa1}
(h-k)|A(h)|^{1/\alpha}\leq
\|(|u|-k)^+\|_{\alpha,A(h)}\leq
\|(|u|-k)^+\|_{\alpha,A(k)}:=I,\quad \forall\alpha\geq 1.
\end{equation}

{\sf Case} $n>2$. Take $\alpha=2^*=2n/(n-2)$ in (\ref{aalfa1}).
Making use of (\ref{sobs})  and $(|u|-k)^+\in
W^{1,q}_{\Gamma_D}(\Omega)$ with $q=2$, and inserting (\ref{nuak}), we deduce
\[
I
\leq S_2
\|\nabla u\|_{2,A(k)}\leq S_2C_{n,p}|A(k)|^{1/2-1/p} .\]
Therefore, we conclude
\[|A(h)|\leq \left(S_2C_{n,p} \over h-k\right)^{2^*}|A(k)|^\beta, \qquad\beta=2^*(1/2-1/p),\]
where $\beta>1$ if and only if $p>n$.
By appealing to \cite[Lemma 4.1]{stamp63-64} we obtain
\[|A(k_0+S_2C_{n,p}|\Omega|^{(\beta-1)/\alpha}2^{\beta/(\beta-1)})|=0.\]
This means that the essential supremmum does not exceed the well determined constant $M:=k_0+S_2C_{n,p}|\Omega|^{(\beta-1)/\alpha}2^{\beta/(\beta-1)}$.

{\sf Case} $n=2$. Choose $\alpha=2$  in (\ref{aalfa1}).
Using (\ref{sobn})  for
$(|u|-k)^+\in
W^{1,2}_{\Gamma_D}(\Omega)$ followed by the H\"older inequality,
and inserting (\ref{nuak}), we obtain
\[
I\leq {1\over\sqrt 2}\|\nabla u\|_{2,A(k)}|A(k)|^{1/2}\leq {C_{2,p}
\over\sqrt 2}|A(k)|^{1-1/p}.
\]
Therefore, we find
\[|A(h)|\leq \left(C_{2,p}/\sqrt{2} \over h-k\right)^{2}|A(k)|^\beta, \qquad\beta=2(1-1/p),\]
where $\beta>1$ if and only if $p>2$.
Then, (\ref{supess}) holds
by appealing to \cite[Lemma 4.1]{stamp63-64} as in the anterior case
($n>2$).

 This completes the proof of Proposition \ref{max}.
\end{proof}

\begin{remark}
The Dirichlet problem  studied by Stampacchia in \cite{stamp63-64}
coincides with (\ref{omega})-(\ref{gama}), with $\Gamma=\emptyset$, $f=g=0$, and $n>2$.
\end{remark}

Let us extend Proposition \ref{max} up to the boundary.
\begin{proposition}\label{maxcor}
Under the conditions of Proposition \ref{max},
 any weak solution $u\in H^1_{\Gamma_D}(\Omega)$ to 
(\ref{omega})-(\ref{gama}) satisfies,
 for $p>2(n-1)$ if $n>2$,  
\begin{eqnarray}
 \label{supesscor}\qquad
{\rm  ess } \sup_{\Omega\cup\Gamma}|u|\leq {\rm  ess } \sup_{\Gamma_D}|g|+
{2^{(n-1)(p-2)/[p-2(n-1)]}\over a_\#}|\Omega|^{1/[2(n-1)]-1/p}
\times\\
\times\left[
\left(|\Omega|^{(n-2)/[2n(n-1)]}S_2+K_2\right)\|{\bf f}\|_{p,\Omega}+
\right.\nonumber\\
\left.+
\left(|\Omega|^{(n-2)/[2n(n-1)]}S_2S_{b'}+K_2S_{p'}\right)
\|f\|_{np/(p+n),\Omega}+\right.\nonumber\\
\left.+
\left(|\Omega|^{(n-2)/[2(n-1)^2]}S_2K_{b'}+K_2K_{p'}\right)
 \|h\|_{(n-1)p/n,\Gamma}
\right],\nonumber
\end{eqnarray}
with $1/b=1/p+(n-2)/[2n(n-1)]$.
For  $n=2$, $p>2\alpha/(\alpha-1)$, and $\alpha>1$, 
then 
 any weak solution $u\in H^1_{\Gamma_D}(\Omega)$
 to (\ref{omega})-(\ref{gama}) satisfies
\begin{eqnarray} \label{supesscor2}
{\rm  ess } \sup_{\Omega\cup\Gamma}|u|\leq {\rm  ess } \sup_{\Gamma_D}|g|+
{2^{\frac{p(\alpha+1)-2\alpha}{p(\alpha-1)-2\alpha}}
\over a_\#}|\Omega|^{{\alpha-1 \over 2\alpha}-{1\over p}}
\times\\
\times\left[\left(|\Omega|^{1/(2\alpha)}S_{2\alpha/(\alpha+2)}+K_
{2\alpha/(\alpha+2)}
\right)\|{\bf f}\|_{p,\Omega}+\right.\nonumber \\
+\left(|\Omega|^{1/2+1/\alpha}S_{2\alpha/(\alpha+2)}
S_{2\alpha p/[p(2\alpha-1)-2\alpha]}+K_
{2\alpha/(\alpha+2)}S_{p'} 
\right)\|f\|_{2p/(p+2),\Omega}\nonumber \\
\left.+\left(|\Omega|^{1/\alpha}S_{2\alpha/(\alpha+2)}
K_{2\alpha p/[p(2\alpha-1)-2\alpha]}+K_
{2\alpha/(\alpha+2)}K_{p'} 
\right)\|h\|_{p/2,\Gamma}
\right].\nonumber
\end{eqnarray}
\end{proposition}
\begin{proof}
Let $k\geq k_0={\rm ess} \sup\{|g(x)|:\ x\in \Gamma_D\}$.
For each $b>2$,
 ${\bf f}\in {\bf L}^b(\Omega)$,  $f\in L^{nb/(b+n)}(\Omega)$, and  $h\in L^{(n-1)b/n}(\Gamma)$,  (\ref{nuak}) reads
\begin{equation}\label{nuak2}
\|\nabla u\|_{2,A(k)\cap\Omega}\leq C_{n,b}|A(k)|^{1/2-1/b} ,
\end{equation}
where $A(k)=\{x\in\bar \Omega:\ |u(x)|>k\}$. 
 With this definition, the integral $I$ from the proof of 
Proposition \ref{max} reads
\[
I=
\|(|u|-k)^+\|_{\alpha,A(k)\cap\Omega}+
\|(|u|-k)^+\|_{\alpha,A(k)\cap\Gamma},
\]
and for  $ h > k > k_0$, we have
\[ 
(h-k)|A(h)|^{1/\alpha}\leq I,\quad \forall\alpha\geq 1.
\] 

{\sf Case} $n>2$. Take $\alpha=2(n-1)/(n-2)<2^*=2n/(n-2)$.
Making use of (\ref{sobs})-(\ref{sobk})  and
$(|u|-k)^+\in
W^{1,q}_{\Gamma_D}(\Omega)$ with $q=2$, we deduce
\begin{equation}\label{isk}
I
\leq S_2|A(k)|^{1/\alpha-1/2^*}\|\nabla u\|_{2,A(k)\cap\Omega}+K_2
\|\nabla u\|_{2,A(k)\cap\Omega}.
\end{equation}
Since there exist different exponents, and our objective is to find one $\beta>1$,
we apply (\ref{nuak2}) twice ($ 1/b=1/p+1/\alpha-1/2^*<1/n$
and $b=p>2(n-1)>n>2$), obtaining
\[
I\leq\left( S_{2} C_{n,b}+K_{2}C_{n,p}\right)
|A(k)|^{1/2-1/p},\qquad 1/b=1/p+(n-2)/[2n(n-1)].
\] 

Therefore, we conclude
\[|A(h)|\leq \left(S_2C_{n,b}+K_{2}C_{n,p}
 \over h-k\right)^{\alpha}|A(k)|^\beta, \qquad\beta=\alpha(1/2-1/p),\]
where $\beta>1$ if and only if $p>2(n-1)$.
Notice that
\begin{eqnarray*}
C_{n,b}\leq |\Omega|^{(n-2)/[2n(n-1)]}\left(
\|{\bf f}\|_{p,\Omega}+S_{b'}\|f\|_{np/(p+n),\Omega}\right)+\\
+|\Omega|^{(n-2)/[2(n-1)^2]}K_{b'}
 \|h\|_{(n-1)p/n,\Gamma}
.
\end{eqnarray*}

{\sf Case} $n=2$.
Using (\ref{sobs}) with $q=2\alpha/(\alpha+2)<2$,
(\ref{sobk}) with $q=2\alpha/(\alpha+1)<2$ and the H\"older inequality, we have
\begin{eqnarray*}
\|(|u|-k)^+\|_{\alpha,A(k)\cap\Omega}&
\leq &S_{2\alpha/( \alpha+2)}
\|\nabla u\|_{2,A(k)\cap\Omega}|A(k)|^{1/\alpha};\\
\|(|u|-k)^+\|_{\alpha,A(k)\cap\Gamma}&
\leq &K_{2\alpha /(\alpha+1)}
\|\nabla u\|_{2,A(k)\cap\Omega}|A(k)|^{1/(2\alpha)} .
\end{eqnarray*}
Thus, we deduce
\[
I
\leq S_{2\alpha/(\alpha+2)}
|A(k)|^{1/\alpha}\|\nabla u\|_{2,A(k)\cap\Omega}+K_{2\alpha/(\alpha+2)}
|A(k)|^{1/(2\alpha)}
\|\nabla u\|_{2,A(k)\cap\Omega}.
\] 

Applying
(\ref{nuak2}) twice ($ b=2\alpha p/(p+2\alpha)>2$
and $b=p>2\alpha/(\alpha-1)>2$), we conclude
\[|A(h)|\leq \left( S_{2\alpha/(\alpha+2)}C_{2,2\alpha p/(p+2\alpha)}+K_
{2\alpha/(\alpha+2)}C_{2,p}
 \over h-k\right)^{\alpha}|A(k)|^{\beta},\]
where $\beta=\alpha(1/(2\alpha)+
1/2-1/p)>1$ if and only if $p>2\alpha /(\alpha-1)$.

Finally, we find (\ref{supesscor})-(\ref{supesscor2})
by appealing to \cite[Lemma 4.1]{stamp63-64}
similarly as to obtain (\ref{supess}).
\end{proof}

Next, let us state the explicit local estimates.
The  Caccioppoli inequality (\ref{cacc})
coincides with the interior  Caccioppoli inequality
whenever $B_R(x)\subset\subset \Omega$ and
 $\eta\in W^{1,\infty}_0(B_R(x))$ denotes
 a cut-off function, and it corresponds to
 \cite[Lemma 5.2]{stamp63-64} if the lower bound of $a$ is related with
 its upper bound by $a_\#=1/a^\#$.
\begin{proposition}\label{max0}
Let $n\geq 2$,   $|\Gamma_D|>0$, ${\bf f} ={\bf 0}$ in $\Omega$,
$f,g,h=0$, respectively, in $\Omega$, on $\Gamma_D$, and on $\Gamma$,
and $u$
be the unique weak solution $u$ to (\ref{omega})-(\ref{gama})
in accordance with Proposition \ref{exist}. 
Then we have

1. the Caccioppoli inequality
\begin{equation}\label{cacc}
\int_{\Omega}\eta^2      |\nabla u|^2\mathrm{dz} \leq 4{a^\#
\over a_\#}
\int_{\Omega}      u^2 |\nabla  \eta|^2\mathrm{dz},
\end{equation}
for any $\eta\in W^{1,\infty}(\mathbb{R}^n)$.

2. For arbitrary  $x\in\Omega$, $R>0$, and
 $k_0\geq 0$, 
  \begin{equation}
\build{{\rm  ess} \sup}_{ \Omega(x,R/2)}|u|\leq k_0+
2^{3n+2+{(3n)^2\over 4}}{ c^{3n/2}\omega_n}
\left( R^{-n}
\int_{\Omega(x,R)}  (| u|-k_0)^2\mathrm{dz}\right)^{1\over 2}
,\label{cota0}
\end{equation}
 where  $c=S_{2n/(n+2)}\left(1+2\sqrt{a^\#/ a_\#} \right)$, and $\Omega(x,r)=\Omega\cap B_r(x)$ for any $r>0$.
\end{proposition}
\begin{proof} 

1.
Let us choose $v=u\eta^2\in  H^1_{\Gamma_D}(\Omega)$
 as a test function in (\ref{pbumax}).
Thus, applying the H\"older inequality we deduce
\[
\int_{\Omega}     a\eta^2 |\nabla u|^2\mathrm{dz}
=-2
\int_{\Omega}     a\eta\nabla u\cdot
\nabla \eta u\mathrm{dz} 
\leq {1\over 2}
\int_{\Omega}     a\eta^2 |\nabla u|^2\mathrm{dz}+
2
\int_{\Omega}     au^2 |\nabla \eta|^2\mathrm{dz}.
\]
Then, using the upper and lower bounds of $a$, we conclude (\ref{cacc}).

2.  Let $x\in\Omega$  be  fixed but arbitrary.
Arguing as in Proposition \ref{max},
let $k\geq k_0$, and 
with the definition of the set $A(k,r)
=\{z\in \Omega(x,r):\ |u(z)|>k\}$,
the property (\ref{aalfa1}) is still valid. In particular,
 we have, for  $ h > k > k_0$, 
\begin{equation}\label{aalfa}
(h-k)|A(h,r)|^{1/2}\leq 
\||u|-k\|_{2,A(h,r)}\leq
\||u|-k\|_{2,A(k,r)},\quad\forall r>0.
\end{equation}

Fix  $0<r<R\leq R_0$, and let us take
 $v={\rm sign}(u)(|u|-k)^+\eta^2\in H^1_{\Gamma_D}(\Omega)$ as a 
test function in (\ref{pbumax}), where $\eta\in W^{1,\infty}(\mathbb{R}^n)
$ is the cut-off
function defined by $\eta\equiv 1$ in $B_r(x)$, $\eta\equiv 0$ in $\mathbb{R}^n
\setminus B_R(x)$, and $\eta(y)=(R-|y-x|)/(R-r)$
for all $y\in B_R(x)\setminus B_r(x)$.
Thus, we have
 $0\leq \eta\leq 1$ in $\mathbb{R}^n$,
 and $|\nabla\eta|\leq 1/(R-r)$ a.e. in $B_R(x)$, and
that (\ref{cacc}) reads
\begin{equation}\label{nuak3}
\int_{A(k,R)}\eta^2 
|\nabla u|^2\mathrm{dz}\leq 4{a^\#\over a_\#} 
\int_{A(k ,R)}   |\nabla\eta|^2    (|u|-k)^2\mathrm{dz}.
\end{equation}

 Making use of
(\ref{sobs}) and $\eta(|u|-k)^+\in W^{1,q}_{\Gamma_D}(\Omega)$ 
with exponent $q=2n/(n+2)<2\leq n$,
 and the H\"older inequality, we have
\begin{eqnarray}\label{aalf}
\||u|-k\|_{2,A(k,r)}\leq \|\eta(|u|-k)^+\|_{2,\Omega}
\leq \\
\leq S_{2n/( n+2)}
\|\nabla (\eta (|u|-k)^+)\|_{2n/(n+2),\Omega}\leq\nonumber\\
\leq S_{2n/( n+2)}
\|(|u|-k)^+\nabla \eta+\eta \nabla u\|_{2,A(k,R)}|A(k,R)|^{1/n} .\nonumber
\end{eqnarray}

Applying 
 the properties of $\eta$,
inserting (\ref{nuak3}) into (\ref{aalf}), and
gathering the second inequality from (\ref{aalfa}), we get
\[
\||u|-k\|_{2,A(h ,r)}
\leq |A(k,R)|^{1/n}{c \over R-r}
\|  |u|-k\|_{2,A(k ,R)}.
\] 

In order to apply \cite[Lemma 5.1]{stamp63-64} that leads
\[
\phi(k_0+2^{{(\alpha+\beta)\beta\over \alpha (\beta-1)}}
{c^{1/\alpha}
\over (\sigma R_0)^{\gamma/\alpha}}[\phi(k_0,R_0)]^{(\beta-1)/\alpha},
R_0-\sigma R_0)=0,\quad \forall \sigma\in ]0,1[\] 
 with $\gamma=1$, $\alpha=2/(3n)$, $\beta=1+2/(3n)>1$,
we use the above inequality, and the  inequality  (\ref{aalfa})
with $r$ replaced by $R$, obtaining
\begin{eqnarray*}
\phi(h,r):=|A(h,r)|
\|  |u|-h\|_{2,A(h ,r)}  \leq\\
\leq |A(k,R)|^{1+{1\over n}-{\alpha\over 2}}
{c \over R-r}
{1 \over (h-k)^{\alpha}}
\|  |u|-k\|_{2,A(k ,R)}^{1+\alpha}\\
=
{c
\over (R-r)  (h-k)^{2/(3n)}}\phi(k,R)^{1+2/(3n)}.
\end{eqnarray*}
Then, taking $R=R_0$
and $\sigma=1/2$,  (\ref{cota0}) holds.

 Therefore, the proof of Proposition \ref{max0} is finished.
\end{proof}

\begin{remark}
The cut-off function explicitly given in Proposition \ref{max0}
does not belong  to 
 $ C^1(B_R(x))$.
\end{remark}

Let us prove the corresponding Neumann version of Proposition \ref{max0}.
\begin{proposition}\label{maxn}
Let $n\geq 2$,   $|\Gamma_D|=0$, ${\bf f} ={\bf 0}$ in $\Omega$,
$f,h=0$, respectively, in $\Omega$,  and on $\Gamma$,
and $u$
be the unique weak solution $u$ to (\ref{omega})-(\ref{gama})
in accordance with Proposition \ref{neumann}. 
 For arbitrary  $x\in\Omega$, $R>0$, and $k_0\in\mathbb{R}$, 
then (\ref{cota0}) holds
with $c=S_{2n/(n+2)}\left(R+1+2\sqrt{a^\#/ a_\#} \right)$.
\end{proposition}
\begin{proof} 
Fix  $k_0\in\mathbb{R}$, $x\in\Omega$,  and $0<r<R\leq R_0$ be  arbitrary.
Arguing as in Proposition \ref{max0}, (\ref{nuak3}) is true by taking
$v=\eta^2 {\rm sign}(u) (|u|-k)^+ - -\hspace*{-0.35cm}\int_{\partial\Omega }
\eta^2 {\rm sign}(u) (|u|-k)^+ \mathrm{ds}\in V_2(\partial\Omega)$ or
 $v=\eta^2 {\rm sign}(u) (|u|-k)^+ - -\hspace*{-0.35cm}\int_{\Omega }
\eta^2 {\rm sign}(u)(|u|-k)^+ \mathrm{dx}\in V_2(\Omega)$ as a 
test function in (\ref{pbumax}), and observing that
$\nabla v=\eta^2 \nabla u
+2\eta\nabla\eta {\rm sign}(u) (|u|-k)^+ 
\in {
\bf L}^2({A(k,R)})$.

Applying 
 the properties of $\eta$, the $W^{1,q}$-Sobolev inequality
for $\eta(|u|-k)^+\in W^{1,q}(\Omega)$ 
with exponent $q=2n/(n+2)<2\leq n$,
 and the H\"older inequality, we have
\begin{eqnarray*}
\||u|-k\|_{2,A(k,r)}\leq \|\eta(|u|-k)^+\|_{2,\Omega}
\leq S_{2n/( n+2)}
\|\eta (|u|-k)^+\|_{1,{2n\over n+2},A(k,R)}\nonumber\\
\leq S_{2n\over n+2}\left((1+{1\over R-r})\||u|-k\|_{2,A(k,R)}
+\|\eta \nabla u\|_{2,A(k,R)}\right)|A(k,R)|^{1/n} .
\end{eqnarray*}
Considering $1+1/(R-r)< (R_0+1)/(R-r)$,
and denoting the new constant by the same symbol $c$,
 we may proceed as in the proof of 
Proposition \ref{max0}. Thus, the proof of 
Proposition \ref{maxn} is complete, taking $R=R_0$  into account.
\end{proof}

\begin{remark}
The set $\Omega(x,R)$ is open and bounded, but  may be neither convex nor 
connexe (see Fig. \ref{domain}).
\end{remark}
\begin{figure}
\resizebox*{1\hsize}{!}{\includegraphics{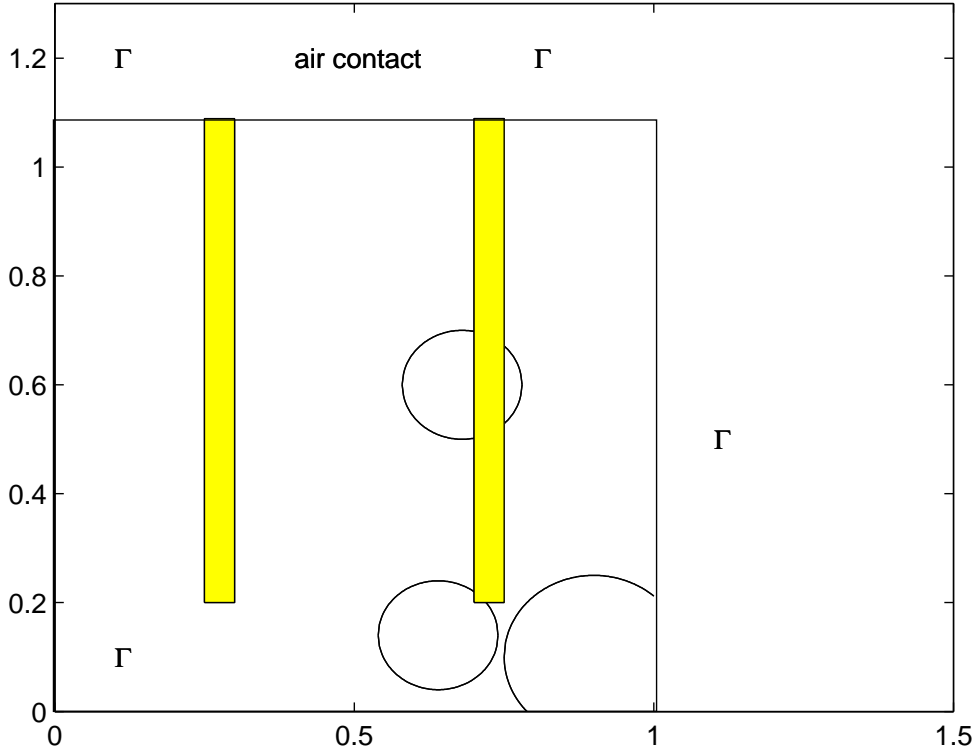}}
\caption{2D schematic representations of a Lipschitz
domain $\Omega$ (with $|\Omega|=1$)
 representing an electrolytic cell,  $\Gamma$ being the union of the 
recipient and air contact boundaries,
 $\Gamma_D$  representing the surface of two electrodes submerged in the
electrolyte, and $\Omega(x,r)=\Omega\cap B_r(x)$ denoting  the subset 
centered at three different points}\label{domain}
\end{figure}

Finally, we state the following local version that will be required in Section 
\ref{secg}. Here the boundary conditions do not play any role,
since one can localize the problem around any point by multiplying
with a suitable cut-off function, and paying for this
by a modified variational formulation.
\begin{proposition}\label{plocal}
Let $n\geq 2$,
 $a\in L^\infty(\Omega)$ satisfies $0< a_\#\leq a\leq a^\#$ a.e. in $\Omega$, 
 $x\in\Omega$, and $R>0$ be such that
 $|\Omega\cap \partial B_R(x)|>0$. If
$u\in H^1(\Omega(x,R))$ solves the local variational formulation
\begin{equation}\label{local}
\int_{\Omega(x,R)}a\nabla u\cdot\nabla v\mathrm{dz}=0,\quad\forall
v\in H^1_{\Omega\cap \partial B_R(x)}(\Omega(x,R)),
\end{equation}
then we have 
 \begin{equation}\label{cotalo}
\build{{\rm  ess} \sup}_{ \Omega(x,R/2)}|u|\leq 
2^{3n+2+{(3n)^2\over 4}}
S_{2n\over n+2}^{3n/2}\left(R+1+2\sqrt{a^\#\over a_\#} \right)^{3n\over 2}
{ \omega_n\over  R^{n/2}}
\left(
\int_{\Omega(x,R)}  u^2\mathrm{dz}\right)^{1\over 2}.
\end{equation}
\end{proposition}
\begin{proof} 
First we
argue as in Proposition \ref{max0}, with   $k_0=0$.
The validity of the properties (\ref{aalfa}) and (\ref{nuak3}) remain.
The application of the $W^{1,2n/(n+2)}$-Sobolev inequality is available
for $\eta(|u|-k)^+\in W^{1,2n/(n+2)}(A(x,R))$.
 Thus, we conclude the proof of 
Proposition \ref{plocal} as in the proof of 
Proposition \ref{maxn}.
\end{proof}

\section{Green kernels}
\label{secg}

In this Section, we reformulate some properties of the Green kernels.
\begin{defn}\label{def2}
For each $x\in\Omega$, we say that $E$ is a Green kernel associated
 to (\ref{omega})-(\ref{gama}), 
if it solves 
\begin{equation}\label{pbud}
 \nabla\cdot(a\nabla E(x,\cdot))=\delta_x,
 \end{equation}
 where $\delta_x$ is the Dirac delta function at the point
$x$, in the following sense: there is $q>1$
such that $E$ verifies the variational formulation
\begin{equation}\label{varf}
\int_{\Omega}     a(y)\nabla_y E(x,y)\cdot
\nabla_y v (y)\mathrm{dy}=v(x),\quad\forall v\in V_q.
\end{equation}
If $|\Gamma_D|>0$, we call it the Green function,
otherwise we call it simply the Neumann function  (also called
 Green function for the Neumann problem or Green function of the second kind),
 and we write  $E=G$ and $E=N$, respectively.
\end{defn}

The existence of the Green function $G$ verifying
\begin{equation}\label{defgreen}
G(x,y)=\lim_{\rho\rightarrow 0}
G^\rho(x,y),\quad
\forall x\in\Omega, \ \mbox{a.e. }y\in\Omega, \ x\not=y,
\end{equation}
is standard if $n>2$  (see for instance \cite{gw,lsw}),
with $G^\rho(x,\cdot)\in H^1_0(\Omega)$ being the unique solution to
\begin{equation}\label{pbgreen}
\int_{\Omega}     a\nabla G^\rho\cdot
\nabla v \mathrm{dy}
={1\over |B_\rho(x)|}\int_{B_\rho(x)}v \mathrm{dy},
\end{equation}
for all $v\in H^1_0(\Omega) $,
for any $x\in\Omega$, and $\rho>0$ such that $B_\rho(x)\subset\Omega$.
Moreover, $G$ satisfies, for some positive 
constant $C(n)$, and $n>2$
 \cite[Theorem 1.1]{gw},
\[
G(x,y)\leq  C(n)(1+\log(a^\#/a_\#)){(a^\#)^{(n-2)/2}\over (a_\#)^{n/2}|x-y|^{n-2}}.
\] 
 In order to explicit the estimates
and simultaneously to extend to $n=2$
and a mixed boundary value problem, let us build the 
Green kernels for $n\geq 2$.
\begin{proposition}\label{green}
Let  $n\geq 2$,  $1\leq  q<n/(n-1)$,
 and $a$ be a measurable (and bounded) function defined in
$\Omega$ satisfying $0<a_\#\leq a\leq a^\#$.
Then, for each $x\in\Omega$ and any $r>0$
such that   $r< {\rm dist}(x,\partial\Omega)$,
 there exists a unique  Green function $G=
G(x,\cdot)\in W^{1,q}_{\Gamma_D}(\Omega)\cap H^1(\Omega \setminus B_r(x))$ 
 according  to Definition \ref{def2},
and  enjoying the following estimates
\begin{eqnarray}
\label{g1}
\|\nabla G\|_{q,\Omega} \leq 
C_1(\Omega,n,q) A +C_2 ( n,q,A);\\
 \| G\|_{qn/(n-q),\Omega} \leq S_{q}\left(
C_1(\Omega,n,q)A+C_2 ( n,q,A)\right),\label{gg}
\end{eqnarray}
with $A= \sqrt{2/ a_\#}$, and
the constants $C_1(\Omega,n,q) $ and $C_2(n,q,A) $ being explicitly 
given in Proposition \ref{W1q}.
Moreover, $G(x,y)\geq 0$ a.e.  $x,y\in \Omega$, and
\begin{equation}
\label{gmax}
|G(x,y)|\leq C(a_\#) \left({\delta(\Omega)\over 2}+1
+2\sqrt{a^\#\over a_\#} \right)^{3n/2}
 |x-y|^{1-n/q},
\end{equation}
for a.e. $ x,y\in \Omega$ such that $x\not= y$, where
\[C(a_\#)=2^{3n+1+{n\over q}+{(3n)^2\over 4}}  S_{2n\over n+2}^{3n/2}\omega_n^{
{3\over 2}+{1\over n}-{1\over q}}
S_{q}\left({
C_1(\Omega,n,q) A}+C_2 ( n,q,A)\right)
.\]
\end{proposition}
\begin{proof}
For any $x\in\Omega$, and $\rho>0$ such that $B_\rho(x)\subset\subset\Omega$,
the existence and uniqueness of 
$G^\rho=G^\rho(x,\cdot)\in H^1_{\Gamma_D}(\Omega)$
solving (\ref{pbgreen}), 
for all $v\in H^1_{\Gamma_D}(\Omega) $, are due to
 Proposition \ref{exist} with  ${\bf f}={\bf 0}$ a.e. in $\Omega$,
$g,h=0$ a.e. on, respectively, $\Gamma_D$ and $\Gamma$,
 and $f=\chi_{B_\rho(x)}/|B_\rho(x)| $ belonging to
$ L^{2n/(n+2)}(\Omega)$ if 
$n>2$, and to $L^{2}(\Omega)$ if $n=2$.
 Moreover, (\ref{dircota}) reads
\begin{equation}\label{grho2}
\|\nabla G^\rho\|_{2,\Omega}\leq{1\over a_\#}\times\left\{
\begin{array}{ll}
 S_2 \omega_n^{1/n-1/2}\rho^{1-n/2}&\mbox{ if } n>2\\
  |\Omega |^{1/2} 
(2\omega_2\rho^2)^{-1/2}&\mbox{ if } n=2
 \end{array}\right. .
\end{equation}
Therefore, for any
 $r>0$ such that $B_r(x)\subset\subset\Omega$,
 there exists $G\in H^1(\Omega\setminus B_r(x))$
such that
\[
G^\rho\rightharpoonup G\quad \mbox{in } H^1(\Omega\setminus B_r(x))\quad
\mbox{ as $\rho\rightarrow  0^+$.}
\]

In order to $G$ correspond to the well defined in (\ref{defgreen}),
the $W^{1,q}$-estimate  (\ref{g1}) is true for $G^\rho$ due to
(\ref{cota1qv}) with $\varkappa=2$, by
applying Proposition \ref{W1q} with ${\bf f}={\bf 0}$,
$g,h=0$, and $f=\chi_{B_\rho(x)}/|B_\rho(x)|\in L^{1}(\Omega)$.
 Then, we can extract a subsequence of $G^\rho,$ still denoted by
$G^\rho,$ 
weakly converging to $G$ in $W^{1,q}(\Omega)$
as $\rho$ tends to $0$, with $G\in V_q$
 solving (\ref{varf}) for all $v\in V_{q'}$.
A well-known property of passage to the weak limit implies (\ref{g1}). 
 The estimate (\ref{gg}) is consequence of the Sobolev embedding
 with continuity constant given in (\ref{sobs}).
 
In order to prove the nonnegativeness assertion, first
calculate
\[
\int_{\Omega}     a|\nabla (G^\rho-|G^\rho|)|^2\mathrm{dy}
={2\over |B_\rho(x)|}
\left(\int_{B_\rho(x)}G^\rho \mathrm{dy}-\int_{B_\rho(x)}|G^\rho| 
\mathrm{dy}\right)\leq 0.
\]
Then, $G^\rho=|G^\rho|$, and by passing to the limit as $\rho$ tends to $0$,
the nonnegativeness claim holds.

For each $x,y\in\Omega$ such that $x\not= y$, we may take  $r<R=|x-y|/2$
such that $G(x,\cdot)\in H^1(\Omega(y,R))$ verifies $\nabla\cdot(a\nabla G)=0$
in $\Omega(y,R)$.
Applying (\ref{cotalo}),  followed by the H\"older inequality
since $qn/(n-q)\geq 2$ means $q\geq 2n/(n+2)$,
we obtain
\begin{eqnarray*}
|G(x,y)|\leq  {C\over R^{n/2}}
\left(
\int_{\Omega(y,R)}  G^2(x,z)\mathrm{dz}\right)^{1/2}\\
\leq C\omega_n^{
{1/ 2}+{1/ n}-{1/ q}}
R^{n\left({1/ 2}+{1/n}-{1/q}\right)-{n/ 2}}
\| G\|_{qn/(n-q),\Omega},
\end{eqnarray*}
with $C=2^{3n+2+{(3n)^2/ 4}} \left(
\partial (\Omega)/2+1+2\sqrt{a^\#
/ a_\#} \right)^{3n/2} S_{2n/( n+2)}^{3n/2}\omega_n$.

Hence, using (\ref{gg}) we conclude (\ref{gmax}),
which completes the proof of Proposition \ref{green}.
\end{proof}

\begin{remark}
Since $qn/(n-q)\rightarrow n/(n-1)$ as $q\rightarrow 1^+$,
and $qn/(n-q)\rightarrow n/(n-2)$ as $q\rightarrow [n/(n-1)]^-$,
the integrability exponent of $G$ in (\ref{gg}) obeys
$n/(n-1)<qn/(n-q)< n/(n-2)$. In conclusion, Proposition \ref{green}
ensures that $G\in L^p(\Omega)$
for any $p\in [1, n/(n-2)[$.
\end{remark}

For each $ x\in\Omega$, the Neumann function 
is defined as $N=G+w$ being the solution of the regularity
 problem \cite[Definition 2.5]{kp},
 where $G\in H^1_0(\Omega)$ is the Green function  
solving (\ref{pbgreen})
and $w\in H^1(\Omega)$, with mean value zero over $\partial\Omega$,
is the unique solution to the variational formulation  \cite[Lemma 2.3]{kp}
\begin{equation}\label{pbneum}
\int_{\Omega}     a\nabla w\cdot
\nabla v \mathrm{dy}
=-\hspace*{-0.4cm}\int_{\partial\Omega}v \mathrm{ds_y}-
-\hspace*{-0.4cm}\int_{\partial\Omega}v \mathrm{d}\omega_\mathrm{y}^x,
\quad\forall v\in V_2(\partial\Omega).
\end{equation}
Here,  $\omega^x$ is the L-harmonic measure \cite{dahl},
i.e. it is unique probability measure on $\partial\Omega$
such that 
\[
L_g(x)=\int_{\partial\Omega}g \mathrm{d}\omega_\mathrm{y}^x,
\]
due to  the Riesz representation theorem applied to the continuous linear 
functional $L:g\in C(\partial\Omega)\rightarrow L_g(x)\in C(\bar\Omega)$,
where $L_g(x)$ is the solution to the Dirichlet problem (\ref{omega})
with ${\bf f}={\bf 0}$ and $f=0$, 
and (\ref{gama}) with $\Gamma_D=\partial\Omega$.
The question of solvability of the
regularity problem  is assigned by the gradient of  the solution
having nontangential limits at almost every point of the boundary \cite{efv,kp}.

\begin{remark}
For each $x\in\Omega$, $N(x,\cdot)$ admits an extension $\widetilde N(x,\cdot)$ across $\partial\Omega$ (cf. 
\cite[Lemmas 2.9 and 2.11]{kp}) to the domain $\widetilde\Omega$ which is such that
\[y\in\partial\widetilde\Omega\Leftrightarrow y\in\mathbb R^n\setminus\bar\Omega:\quad
y=y^*+(y^*-{\mathcal T}(y^*)), \quad\mbox{ for some }y^*\in\partial\Omega,
\]
where $\mathcal T$ is the homothety function that reduces $\partial\Omega$ into its half, i.e.
the homothetic boundary with measure $|\partial\Omega|/2$.
 That is, each $y\in\widetilde\Omega\setminus\bar\Omega$
is the reflection of $\mathcal T(y^\delta)$ across
$\partial\Omega$ in the following sense: 
\[y=y^*+(y^*-y^\delta){y^*-\mathcal T(y^*)\over |y^*-\mathcal T(y^*)|},\]
where $y^\delta\in\Omega$ is such that
 $y^\delta=y^*-\delta(y^*-\mathcal T(y^*))/|y^*-\mathcal T(y^*)|$, for some  $0<\delta<\delta(\Omega)$.
\end{remark}
Since our concern is on
 weak solutions to (\ref{omega})-(\ref{gama}) in accordance
with Definition \ref{def1}, 
 we reformulate for $n\geq 2$ the existence result due to  Kenig 
and Pipher on solutions to the Neumann problem in
bounded Lipschitz domains if $n>2$, with no information of its boundary behavior.
\begin{proposition}\label{neum}
Let  $n\geq 2$, $1\leq q<n/(n-1)$,
 and $a$ be a measurable (and bounded) function defined in
$\Omega$ satisfying $0< a_\#\leq a\leq a^\#$.
Then, for each $x\in\Omega$,
 there exists a Neumann function $N=N(x,\cdot)\in V_{q}$ solving (\ref{pbud})
 that satisfies (\ref{g1})-(\ref{gg}), and (\ref{gmax}),
with  $A=2/\sqrt{a_\#} $.
\end{proposition}
\begin{proof}
For each $x\in\Omega$,  and $\rho>0$ such that $B_\rho(x)\subset\Omega$,
the existence of a unique Neumann function 
$N^\rho(x,\cdot)\in V_2$ solving (\ref{pbgreen}), 
for all $v\in V_2$,  is consequence of
 Proposition \ref{neumann} with  ${\bf f}={\bf 0}$,
$g,h=0$, and $f=\chi_{B_\rho(x)}/|B_\rho(x)| \in L^{t}(\Omega)$
for $t=2n/(n+2)$ if $n>2$, and any $t<2$ if $n=2$.
Arguing as in the proof of Proposition \ref{green},
$N^\rho$ belongs to $W^{1,q}(\Omega)$, uniformly for $x\in\Omega$,
 according to (\ref{cota1qv}) with $\varkappa= 4$. Therefore,
 we may pass to the limit as $\rho\rightarrow 0$,
finding $N\in V_{q}$ solving (\ref{pbud}).
The remaining estimates (\ref{gg})-(\ref{gmax}),
under $A=2/\sqrt{a_\#} $, are obtained
exactly as in the proof of Proposition \ref{maxn}.
\end{proof}

Hereafter,
 $\partial_{x_i}$ denotes the partial derivative $\partial/\partial x_i$. 
\begin{proposition}\label{ddg}
Let   $n\geq 2$, $1\leq q<n/(n-1)$,
 $E$ be the symmetric function that is either
the Green function $G$ or 
the Neumann function $N$
in accordance with Propositions \ref{green} and  \ref{neum}, respectively.
If $a\in L^\infty(\Omega)$ verifies $0< a_\#\leq a\leq a^\#$ a.e. in $\Omega$,
 then for every $i=1,\cdots,n$
 $\partial_{x_i} E(x,\cdot)\in V_q\cap L^{\infty}(\Omega)$ is
 uniformly bounded for $x\in\Omega$. In particular, it
 satisfies (\ref{g1})-(\ref{gg}), and (\ref{gmax}),
 where  $A=\sqrt{\varkappa /a_\#}$,
with $\varkappa=2$  if $|\Gamma_D|>0$, and $\varkappa=4$  if $|\Gamma_D|=0$.
 \end{proposition}
\begin{proof} 
For each $x\in \Omega$, we may approximate $\partial_{x_i} E(x,\cdot)$ by
$\{\partial_{x_i} E^\rho\}_{\rho>0}$, where
$E^\rho=E^\rho(x,\cdot)\in V_{2}$ solves (\ref{pbgreen})
for every  $\rho>0$ such that $B_\rho(x)\subset\Omega$.
Since $\partial_{x_i} [\chi_{B_\rho (x)}/(\rho^n |B_1(0)|)]$ 
is a Dirac delta function,
Proposition \ref{dirac1} ensures that $\partial_{x_i} E^\rho$
verifies (\ref{g1}),  and
also (\ref{gg}) by the Sobolev inequality (\ref{sobs}), 
with  $A=\sqrt{\varkappa /a_\#}$,
where $\varkappa=2$  if $|\Gamma_D|>0$, and $\varkappa=4$  if $|\Gamma_D|=0$.
Consequently, (\ref{g1})-(\ref{gg}) hold, by passage to the weak limit.

To prove the estimate (\ref{gmax}) for $\partial_{x_i}E(x,\cdot)$,
 let us take $y\in\Omega$  such that  $R=|x-y|/2>0$. Thus, 
 $E(x,\cdot)\in H^1(\Omega(y,R))\cap V_q$ verifies $\nabla\cdot(a\nabla 
\partial_{x_i} E)=0$
in $\Omega(y,R)$, for every $i=1,\cdots ,n$.
Therefore,
 we proceed by using
the  argument already used in the  proof of  Proposition \ref{green},
with $G$ replaced by $\partial_{x_i}E$.
\end{proof} 

\begin{remark}
Notice that $q'>n$ implies that $E$ is not an admissible test function in
\[
 \int_{\Omega}     a(y)\nabla  \partial_{x_i}
E(x,y)\cdot\nabla v(y) \mathrm{dy}=  \partial_{x_i}
v(x),\quad\forall v\in V_{q'} ,
\]
for each $x\in\Omega$, and for every $i=1,\cdots ,n$,
which comes from Definition \ref{def2}, i.e. due to differentiate (\ref{varf})
 under the integral sign  in $x_i$.
We emphasize that for each $x\in\Omega$ and any $r>0$
such that   $r< {\rm dist}(x,\partial\Omega)$, the symmetric function
$E(x,\cdot)\in V_q\cap H^1(\Omega \setminus B_r(x))$ 
 verifies, by construction, the limit system
 of identities
 \[
\int_{\Omega} a\varphi^2 \nabla (\nabla_x E)(x,\cdot)\cdot\nabla v\mathrm{dz} =-2
\int_{\Omega}   a   \varphi v\nabla ( \nabla_x E)(x,\cdot)\cdot\nabla 
\varphi \mathrm{dz},
\] 
for any $\varphi\in W^{1,\infty}(\mathbb{R}^n)$
 such that supp$(\varphi)\subset \mathbb{R}^n\setminus
\overline{B_{\partial(\Omega)}(x)\cup B_{r}(x)}$, and
 for all $v\in V_1\cap  H^1(\Omega\cap \mathrm{supp}(\varphi))$.
\end{remark}

Next, 
we prove additional estimates for the derivative  of the weak solution to
(\ref{omega})
with ${\bf f}={\bf 0}$ and $f=0$, 
if we strengthen the hypotheses on
the regularity of the coefficient $a$. Indeed we proceed as in \cite{gw} where
the coefficient is assumed Dini-continuous to be enable 
to derive some more pointwise estimates for the derivative of the Green
kernels.
\begin{proposition}\label{propu}
Let $a\in L^\infty(\Omega)$ satisfy $0< a_\#\leq a\leq a^\#$ a.e. in $\Omega$.
If  there exists a function $\omega:[0,\infty[\rightarrow [0,\infty[$ such that,
a.e. $x, y\in\Omega$,
\begin{equation}
|a(x)-a(y)|\leq \omega(|x-y|),
\quad
0< C_a:=\int_0^1{\omega(t)\over t}\mathrm{dt}<\infty\label{ar}
\end{equation}
 then for each $x\in\Omega$, and $R>0$,
any function $u\in W^{1,1}(\Omega)$ solving
\begin{equation}\label{pbua}
 \nabla\cdot(a\nabla u)=0\mbox{ in }\Omega(x,R),
 \end{equation}
in the sense of distributions,  enjoys a.e. $y\in\Omega$,
\begin{equation}\label{nu}
|\nabla u(y)|\leq
{a_\#\delta(\Omega)\over C_an\omega _{n}} \left({4\pi\over 3}+n\right)
 {1\over |x-y| }
\int_{B_{d}(y)}{|u(z)|\over |y-z|^n}\mathrm{dz},
\end{equation}
where
\begin{equation}
d=\left\{\begin{array}{ll}
|x-y|/2,&\mbox{if }|x-y|<2r\\
|x-y|/\nu,&\mbox{if }|x-y|=\nu r
\end{array}\right.
\label{defd}\end{equation}
for some  $2\leq\nu<\delta(\Omega)/r$ and 
 $0<r<\min\{1,\delta(y)\}$ with $\delta(y):={\rm dist}(y,\partial\Omega)$.
\end{proposition}
\begin{proof} 
By density, since $u\in W^{1,1}(\Omega)$ there exists a sequence $\{u_m\}_{m\in\mathbb N}\subset C^1(\bar{\Omega})$ such that $u_m\rightarrow u$
in $W^{1,1}(\Omega)$. In particular, $u_m\rightarrow u$ in $L^1(\Omega)$
and $\nabla u_m\rightarrow \nabla u$ a.e. 
in  $\Omega$. Thus, it is sufficient to prove the estimate (\ref{nu}),
under the assumption  $u\in C^{1}(\bar\Omega)$.

Fix  $x\in\Omega$, and $R>0$.
For an arbitrary $y\in \Omega(x,R)$ we can choose 
 $0<r<\min\{R,1,\delta(y)\}$ and $M>0$ such that
\begin{equation}\label{nui}
\sup_{z\in B_r(y)}
|x-z||\nabla u(z)|=M\quad\mbox{ and }\quad |x-y||\nabla u(y)|>2bM,
\end{equation}
for some constant $b\in ]0,1/2[$.
Since $\Omega$ is bounded, we can take  $2\leq\nu<\delta(\Omega)/r$
and define $d$ as in (\ref{defd}).
Notice that $d\leq r$ implies $B_d(y)\subset \subset \Omega$.

In order to determine the final constant in (\ref{nu}),
let $\eta\in C^1_0(B_d(y))\cap W^{2,\infty}(\Omega)
$ be the cut-off function explicitly given by
\[\eta(z)=
\left\{\begin{array}{ll}
1,&\mbox{ if }z\in B_{d/2}(y)\\
2^{-1}\left(1+
\cos\Big[{4\pi\over 3d^2}
 (|z-y|^2-d^2/4)\Big]\right),&\mbox{ if }d/2\leq |z-y|<d\\
0,&\mbox{ if }z\in \Omega\setminus B_{d}(y).
\end{array}\right.
\] 
Thus, $\eta$ satisfies $0\leq \eta\leq 1$,
\begin{eqnarray}\label{c1}
|\nabla\eta(z)|\leq {c_1/ d}\leq c_1\nu
|x-y|^{-1},&\forall z\in\Omega & (c_1={4\pi/ 3})\\
\qquad|\Delta\eta(z)|\leq {c_1c_2/ d^2}
\leq c_1c_2\nu^2
|x-y|^{-2},&\mbox{a.e. } z\in\Omega &(c_2=c_1+n).\label{c2}
\end{eqnarray}
For $w\in B:=B_d(y)$, we multiply (\ref{pbua}) by $G_L(w,\cdot)\eta/a(y)$
where $G_L$ is the  fundamental solution of Laplace equation, 
\[
G_L(w,z)=
C_L\left\{\begin{array}{ll}
(2-n)^{-1}|w-z|^{2-n}&\mbox{if }n>2\\
\ln[|w-z|]&\mbox{if }n=2
\end{array}\right.\]
with 
\begin{equation}\label{gl}
C_L:={ba_\#r\over (c_1\nu+2(n+1)) C_an\omega_{n}},
\end{equation} 
and we integrate over $B$ to get
\begin{eqnarray*}
0=\int_B{a(z)\over a(y)}\nabla_z u(z)\cdot \nabla_z(\eta(z)G_L(w,z))
\mathrm{dz}\\
=\int_B
{a(z)-a(y)\over a(y)}\nabla_z u(z)\cdot \nabla_z(\eta(z) G_L(w,z))\mathrm{dz}-
\int_B u\Delta_z \eta G_L (w,\cdot)\mathrm{dz}\\
-2\int_B u(z)\nabla_z\eta(z) \cdot\nabla_z G_L(w,z) \mathrm{dz}-u(w)\eta(w),
\end{eqnarray*}
taking into account the use of integration by parts.
Differentiating the above identity
with respect to $w$ and setting $w=y$ it results in
\[\nabla u(y)=I_1+I_2,\]
where
\begin{eqnarray*}
I_1=&\int_B{a(z)-a(y)\over a(y)}\Big( \nabla_yG_L(y,z)\nabla_z\eta(z)+
\eta(z)\nabla_y\nabla_zG_L(y,z)\Big)\cdot \nabla_z u(z)\mathrm{dz};\\
I_2=&-\int_B u(z)\nabla_yG_L (y,z) \Delta_z \eta(z) \mathrm{dz}
-2\int_B u(z)\nabla_y\nabla_zG_L(y,z)\cdot\nabla_z\eta(z)\mathrm{dz}.
\end{eqnarray*}
Using the lower bound of $a$, the definition of $G_L$, and the properties of $\eta$, we have
\begin{eqnarray*}
I_1\leq{C_L\over a_\#} \int_B|a(z)-a(y)|\left( {c_1\nu
\over |y-z|^{n-1}|x-y|}+{n+1\over |y-z|^n}\right)|\nabla_z u(z)|\mathrm{dz};\\
I_2\leq C_L c_1 \nu
\int_B |u(z)|\left( {c_2\nu 
\over |y-z|^{n-1}|x-y|^2}+{2(n+1)\over |y-z|^n|x-y|}
\right)\mathrm{dz}.
\end{eqnarray*}
By appealing to (\ref{nui}), we obtain 
\begin{eqnarray*}
2bM<|x-y||\nabla u(y)|\leq\\
\leq{C_LM\over a_\#} \int_B{|a(z)-a(y)|\over |x-z|}
\left( {c_1\nu\over |y-z|^{n-1}} 
+(n+1){|x-y|\over |y-z|^n}\right)\mathrm{dz}+\\
+C_L c_1 \nu\int_B |u(z)|\left( {c_2\nu 
\over |y-z|^{n-1}|x-y|}+{2(n+1)\over |y-z|^n}
\right)\mathrm{dz}.
\end{eqnarray*}

Considering that, for all $x,y\in\Omega$ and $z\in B_d(y)$, 
\[|y-z|\leq 
|x-y|\leq |x-z|+|z-y| 
\]
we obtain
\begin{eqnarray}\label{mm}
2bM<{C_LM\over a_\#} \left(\int_B{|a(z)-a(y)|} {c_1\nu+n+1
\over |y-z|^{n-1}|x-z|}\mathrm{dz}+\right.\\
\left.+ (n+1)\int_B{|a(z)-a(y)|\over |y-z|^n}\mathrm{dz}\right) 
+C_L c_1\nu \left(c_2\nu+2(n+1)
\right)\int_B{ |u(z)|\over |y-z|^{n}}
\mathrm{dz}.\nonumber
\end{eqnarray}

Let us analyze the first integral of RHS in (\ref{mm}). From the definition of the radius $d$,
we consider two different cases: $|x-y|=\nu r$ and otherwise.
In the first case, from $z\in B_d(y)$ we have $|y-z|<|x-y|/\nu$. Hence, we find
$(\nu-1)|y-z|<|x-z|$ and consequently
\begin{equation}\label{yz}
{1\over |y-z|^{n-1}|x-z|}<{1\over (\nu-1) |y-z|^{n}}\leq {1\over |y-z|^n},\qquad \nu\geq 2.
\end{equation}
If $d=|x-y|/2$ and $z\in B_d(y)$,  clearly  (\ref{yz}) holds denoting $\nu=2$.

Returning to (\ref{mm}), substituting the value of $C_L$
from (\ref{gl}) with $r\leq 1$, and dividing by $b>0$, we write it as
\begin{equation}\label{mma}
2M<{M\over C_an\omega_{n}} \int_B{|a(z)-a(y)|\over |y-z|^n}\mathrm{dz}
+{a_\#\delta(\Omega)c_1\over C_an\omega_{n}}\left(1+{n\over c_1}\right)
\int_B{ |u(z)|\over |y-z|^{n}}
\mathrm{dz}.
\end{equation}

In a $n$-dimensional Euclidean space,
the spherical coordinate system consists of a radial coordinate $t$,
  and $n-1$ angular coordinates $\phi_1,\cdots, \phi_{n-2}\in [0,\pi]$,
and $\phi_{n-1}\in [0,2\pi[$, and
 the Cartesian coordinates are 
 $z_1=y_1+t\cos(\phi_1)$,
 $z_2=y_2+t\sin(\phi_1)\cos(\phi_2)$,
$ \cdots,$  $z_{n-1}=y_{n-1}+
t\sin(\phi_1)\cdots \sin(\phi_{n-2})\cos(\phi_{n-1})$,
and $z_{n}=y_n+t\sin(\phi_1)\cdots
\sin(\phi_{n-1})$.
Since the Jacobian of this transformation is
 $t^{n-1}\sin^{n-2}(\phi_1) \sin^{n-3}(\phi_{2})
\cdots\sin(\phi_{n-2})$, and
\[
\omega_n={1\over n}
\int_0^\pi\cdots \int_0^\pi\int_0^{2\pi}
\sin^{n-2}(\phi_1) \sin^{n-3}(\phi_{2})\cdots\sin(\phi_{n-2})\mathrm{d\phi_1}
\cdots\mathrm{d\phi_{n-1}},
\]
applying (\ref{ar}),  we deduce
\begin{equation}\label{ardini}
\int_B{|a(z)-a(y)|\over |y-z|^n}\mathrm{dz}
\leq \int_B{ \omega(|y-z|)\over |y-z|^{n}}\mathrm{dz}=
n\omega_n\int_0^d {\omega(t)\over t}\mathrm{dt}
\leq C_an\omega_n.
\end{equation}
Inserting this last inequality into (\ref{mma}), 
 we find (\ref{nu}).
\end{proof}

\begin{remark}
Observing (\ref{ardini}), the assumption (\ref{ar}) can be replaced by $a$
belonging to
the VMO space of vanishing mean oscillation
 functions which is constituted by the functions $f$ belonging to the BMO space such that verify
\[\lim_{r\rightarrow 0} \sup_{\rho\leq r}
-\hspace*{-0.45cm}\int_{B_\rho}|f(x)-
\Big( -\hspace*{-0.45cm}\int_{B_\rho}f(y)\mathrm{dy}ç\Big)|\mathrm{dx}
=0,\]
where $B_\rho$ ranges in the class of the balls with radius $\rho$ contained in $\Omega$. We recall that 
the John-Nirenberg space BMO of the functions of bounded mean
oscillation is defined as
\[BMO=\{f\in L^1_{\rm loc}(\Omega):\ \sup_B
-\hspace*{-0.45cm}\int_B|f(x)-
\Big( -\hspace*{-0.45cm}\int_Bf(y)\mathrm{dy}\Big)
\mathrm{dx}<\infty\},
\]
where $B$ ranges in the class of the balls contained in $\Omega$.
\end{remark}

\begin{remark}
The upper bound in (\ref{nu}) is not optimal, it depends on the choice of the 
cut-off function through the contants $c_1$ and $c_2$ (cf. 
(\ref{c1})-(\ref{c2}) and (\ref{mma})).
\end{remark}
 
\begin{proposition}\label{pg}
Let   $n\geq 2$, $1\leq q<n/(n-1)$,
$E$ be the symmetric function that is either the Green function $G$ or 
the Neumann function $N$
in accordance with Propositions \ref{green} and  \ref{neum}, respectively.
If $a\in L^\infty(\Omega)$ satisfies $0< a_\#\leq a\leq a^\#$ a.e. in $\Omega$,
and (\ref{ar}), then a.e. $ x,y\in \Omega$,
\begin{equation}
\label{g2} 
|\nabla_y E(x,y)|\leq 
C ( \Omega,n,q,a) |x-y|^{-n/q}, 
\end{equation}
with 
\begin{eqnarray*}
C(\Omega,n,q,a)= {\delta(\Omega) \over C_a}
(4 {\pi\over 3}+n)2^{3n+2{n/ q}+{(3n)^2/ 4}}  S_{2n\over n+2}^{3n/2}
\omega_n^{{1/ n}-{1/ q}+{3/ 2}}
S_{q}\times\\
\times
\left(C_1(\Omega,n,q) \sqrt{\varkappa a_\#}+a_\#
C_2 ( n,q, \sqrt{\varkappa\over a_\#})\right)\left({\delta(\Omega)\over 2}+1+2\sqrt{a^\#\over a_\#}
 \right) ^{3n\over 2},
 \end{eqnarray*}
 where $\varkappa=2$  if $|\Gamma_D|>0$, $\varkappa=4
$  if $|\Gamma_D|=0$, and the constants $C_1(\Omega,n,q) $ and $C_2(n,q,
\sqrt{\varkappa /a_\#}) $ are explicitly 
given in Proposition \ref{W1q}.
\end{proposition}
\begin{proof} 
Let $x\in\Omega$ be arbitrary. 
Using the property  (\ref{nu}), and applying (\ref{gmax}), we get
\begin{eqnarray*}
|\nabla_y E(x,y)|\leq { \delta(\Omega)\over C_a}
\left(C_1(\Omega,n,q) \sqrt{\varkappa a_\#}+a_\#
C_2 ( n,q, \sqrt{\varkappa\over a_\#})\right)
{ C \over |x-y| }\times\\
\times
\left({\partial(\Omega)\over 2}+1+2\sqrt{a^\#\over a_\#} \right) ^{3n/2}
\int_{B_{d}(y)}{|x-z|^{1-n/q}\over |y-z|^n}\mathrm{dz},
\end{eqnarray*}
with
\[
C=(4
{\pi\over 3}+n)2^{3n+1+{n/ q}+{(3n)^2/4}}  S_{2n/( n+2)}^{3n/2}{\omega_n^{
{1/ n}-{1/ q}+{1/ 2}}\over n}
S_{q}
.\]
Considering that, for all $x,y\in\Omega$ and $z\in B_d(y)$ 
with $d\leq |x-y|/2$,
\[|y-z|\leq 
{|x-y|\over 2}\leq {1\over  2}(|x-z|+|z-y|)\Longrightarrow
\left\{\begin{array}{l}
 |y-z|\leq |x-z|\\
|x-y|\leq 2|x-z|
\end{array}\right., 
\]
we compute
\[\int_{B_{d}(y)} {|x-z|^{1-{n\over q}}\over
|y-z|^n}\mathrm{dz}\leq 2^{n/q}\int_{B_{d}(y)} {|x-y|^{-n/q}\over
|y-z|^{n-1}}\mathrm{dz}
=2^{n/q} dn\omega_{n}|x-y|^{-n/q},
\]
where the Riesz potential is calculated by the spherical transformation
 as in the above proof.
Next,
from $d\leq |x-y|/2$ we find (\ref{g2}).\end{proof}

\section{$W^{1,p}$-constants ($p>n$)}
\label{reg}

Let $p>n$, $g=0$ on $\Gamma_D$ (possibly empty), and $u\in V_p$
solve (\ref{pbumax}) for all $v\in V_{p'}$.
Its existence depends on several factors.

The regularity theory  for solutions of the class of divergence form
elliptic equations in convex domains  guarantees the
 existence of a unique strong solution 
 if
the coefficient   is uniformly continuous,
taking the Korn perturbation method \cite[pp. 107-109]{grisv} into account.
This result can be proved if the convexity of $\Omega$ is replaced by weaker assumptions, 
for instance
 when $\Omega$ is a plane bounded domain
with Lipschitz and piecewise $C^2$ boundary whose angles are all convex 
\cite[p. 151]{grisv},
or when $\Omega$ is a plane bounded domain
with curvilinear polygonal $C^{1,1}$ boundary whose angles are all strictly convex \cite[p. 174]{grisv}.
For general bounded domains  with Lipschitz  boundary,
the higher integrability of the exponents for the gradients
 of the solutions may be  assured \cite{agranovich,sav},
under particular restrictions on the coefficients.
In \cite{elschk,haller}, the authors figure out configurations of
(discontinuous) coefficient functions and geometries of the domain,
such that the required result does hold. 
In \cite{lv}, the authors derive global
$W^{1,\infty}$ and piecewise $C^{1,\alpha}$ estimates
 with piecewise H\"older continuous coefficients, which
depend on the shape and
on the size of the surfaces of discontinuity of the coefficients,  but they are independent
of the distance between these surfaces. 
When the coefficient  of the principal part of the divergence form elliptic
equation is 
only supposed to be bounded and measurable,
Meyers extends Boyarskii
result to n-dimensional elliptic equations of divergence structure \cite{mey}.
Adopting this rather weak hypothesis,
the works \cite{gro89,groreh,ott} extend to mixed boundary value problem the
result due to Meyers.

For a domain of class $C^{1,1}$, 
$W^{1,p}$-regularity of the solution is found for $1<p<\infty$
 in \cite{dif,rag99}
under the hypotheses on
the coefficients of the principal part are to belong to the Sarason class
\cite{sara} of vanishing
mean oscillation functions (VMO).
In \cite{geng}, the author extends the $W^{1,p}$-solvability 
 to the Neumann problem  
for a range of integrability exponent
$p\in ]2n/(n+1)-\varepsilon,2n/(n-1)+\varepsilon[$, where $\varepsilon>0$
depends on $n$, the ellipticity constant,
 and the Lipschitz character of $\Omega$.
 Notwithstanding, the results concerning VMO-coefficients
 are irrevelant for real world applications. The reason is that
the VMO-property forbids jumps across a hypersurface,
what is the generic case of discontinuity.
 
 For Lipschitz domains with small Lipschitz constant, the Neumann problem
 is solved in \cite{dong}, where the
leading coefficient is
 assumed to be measurable in one direction, to have small BMO
semi-norm in the other directions, and to have small BMO semi-norm in a
neighborhood of the boundary of the domain.
We refer to  \cite{bw2004}
for the optimal $W^{1,p}$ regularity theory regarding Dirichlet problem 
on bounded domains whose boundary is so rough that the unit normal vector
is not well defined, but is well approximated by hyperplanes at every point and at every scale (Reifenberg flat domain);
and the coefficient belongs to
 the space $\mathcal V$  such that $C(\Omega)\subset$ VMO $\subset \mathcal{V}\subset$ BMO which
is defined as  the BMO space with their BMO semi-norms sufficiently
small.
In \cite{bw2005}
the authors obtain the global $W^{1,p}$ regularity theory  a linear
elliptic equation in divergence form with the conormal boundary condition via perturbation
theory in harmonic analysis and geometric measure theory, in particular
on maximal function approach.

Let us begin by establishing the relation between
any weak solution $u\in V_p$ ($p>n$)
and the  Green kernel $E$ associated to (\ref{omega})-(\ref{gama}),
i.e. $E\in V_{p'}$ is either the Green or the Neumann functions,
 $E=G$ and  $E=N$, in accordance with Propositions
\ref{green} and \ref{neum}, respectively.
To this end, we take $v=E\in V_{p'}$ and $v=u\in V_{p}$ as test functions
in (\ref{pbumax}) and (\ref{varf}), respectively,
obtaining  the Green representation formula
\begin{equation}\label{rep}
u(x)=\mathcal{T}({\bf f})(x)+\mathcal{S}(f)(x)+\mathcal{K}(h)(x),\qquad 
 x\in \Omega,
\end{equation}
where  $\mathcal{T}$, $\mathcal{S}$, and $\mathcal{K}$
are the layer 
potential operators defined by
\begin{eqnarray*}
\mathcal{T}({\bf f})& =&
\sum_{j=1}^n\int_{\Omega}f_j(y){\partial E\over\partial{y_j} }(\cdot
,y)\mathrm{dy} ; \\
\mathcal{S}(f) &=&
\int_\Omega f(y)E(\cdot,y)\mathrm{dy};\\
\mathcal{K}(h)&=&
\int_{\partial\Omega}
\chi_\Gamma(y) h(y)E(\cdot,y)\mathrm{ ds_y}.
 \end{eqnarray*}

For every $0<\lambda<n$, $u\in L^s(\mathbb{R}^n)$,
  $v\in L^t(\mathbb{R}^n)$, with $s,t>1$ and $\lambda/n+1/s+1/t=2$,
the Hardy-Littlewood-Sobolev inequality in its general form states the following:
\begin{equation}\label{hardy}
\int_{\mathbb{R}^n}
\int_{\mathbb{R}^n}u(x)|x-y|^{-\lambda}v(y) \mathrm{dx}\mathrm{ dy}
\leq C(n,s,\lambda)
\|u\|_{s,\mathbb{R}^n}
\|v\|_{t,\mathbb{R}^n},
\end{equation}
where the constant is sharp \cite{lieb}, if $s=t=2n/(2n-\lambda)$, defined by
\begin{eqnarray*}
C(n,\lambda)=
 \pi^{\lambda/2}
{\Gamma((n-\lambda)/2)\over \Gamma(n-\lambda/2)}
\left[
{\Gamma (n)\over\Gamma(n/2)}\right]^{1-\lambda/n}.
\end{eqnarray*}

In the presence of the Hardy-Littlewood-Sobolev inequality,
we prove the following $W^{1,p}$-estimate.
\begin{proposition}\label{ppf}
Let  $p>1$,  $f\in L^t(\Omega)$ with $t\in ]pn/(p+n),p[$,
${\bf f}={\bf 0}$ in $\Omega$, $g=0$ on $\Gamma_D$
(possibly empty),
$h=0$ on $\Gamma$, $a\in L^\infty(\Omega)$   
satisfy $0< a_\#\leq a\leq a^\#$ a.e. in $\Omega$, and (\ref{ar}).
 If $u\in V_p$ solves (\ref{pbumax}), for all $v\in V_{p'}$,
then  $u$ satisfies
 \begin{eqnarray}
\label{nupf}
\|\nabla u\|_{p,\Omega} \leq
C(n,p',n/q) C(\Omega,n,q,a)\|f\|_{t,\Omega},
\end{eqnarray}
with   $1/q=1+1/p-1/t$, $C(n,p',n/q)$ relative to (\ref{hardy}), and 
$C(\Omega,n,q,a)$ determined in Proposition \ref{pg}.
In particular, for $2<p<2n/(n-1)$ we have
\[\|\nabla u\|_{p,\Omega} \leq
C(n,2n/p) C(\Omega,n,p/2,a)\|f\|_{p',\Omega}.
\]
\end{proposition}
\begin{proof}
Since $\nabla u\in {\bf L}^p(\Omega)$,  (\ref{rep}) holds.
Differentiating it, for $i=1,\cdots,n$, we deduce
\[{\partial u\over\partial x_i }(x)=
\int_{\Omega} {\partial E\over\partial x_i}(x,y)
{f(y)}\mathrm{dy}.
\] 
Let ${\bf w}\in {\bf L}^{p'}(\Omega)$ be arbitrary such that $\|{\bf w}\|_{p',\Omega}=1$.
 Using (\ref{g2}) for any $1<q<n/(n-1)$, and applying the Fubini-Tonelli Theorem,
 we find
\begin{eqnarray*}
\int_\Omega\nabla u(x)\cdot{\bf w}(x)\mathrm{dx}\leq C(\Omega,n,q,a)
\int_{\Omega}\int_\Omega  |{\bf w}(x)| |x-y|^{-n/q} |f(y)| 
\mathrm{dx}\mathrm{dy}.
\end{eqnarray*}
Next, using (\ref{hardy}) with $\lambda=n/q$, $s=p'$, and $1/t=1+1/p-1/q$, 
we  conclude (\ref{nupf}).

For the particular situation, we choose $1<q=p/2<n/(n-1)$
 and we use (\ref{hardy}) with $s=t=p'$.
\end{proof}

Having 
the results established in Section \ref{secg} in mind, we find a 
$W^{1,p}$-estimate  for weak solutions
 where  the regularity (\ref{ar}) of the leading coefficient
is not a necessary condition.
\begin{proposition}\label{ppv}
Let  $p>n$, 
${\bf f}\in{\bf L}^p(\Omega)$, $f\in L^p(\Omega)$, $g=0$ on $\Gamma_D$
(possibly empty),
$h\in L^p(\Gamma)$, $a\in L^\infty(\Omega)$   
satisfy $0< a_\#\leq a\leq a^\#$ a.e. in $\Omega$,
 and $u\in V_p$ solve (\ref{pbumax}), for all $v\in V_{p'}$.
Then  $u$ satisfies
 \begin{eqnarray*} 
\|\nabla u\|_{p,\Omega} \leq|\Omega|^{1/p}\left({
C_1(\Omega,n,p') \over 
a_\#}+C_2 ( n,p',{1\over a_\#})\right)\times\\
\times(\|{\bf f}\|_{p,\Omega}+ \|f\|_{p,\Omega}+ |\Gamma|^{1/[p/n-1)]} K_{p'}
\|h\|_{p,\partial\Omega}),\nonumber
\end{eqnarray*} 
with the constants $C_1(\Omega,n,p') $ and $C_2(n,p',1/a_\#) $ being explicitly 
given in Proposition \ref{W1q}.
\end{proposition}
\begin{proof}
Differentiating (\ref{rep}), for $i=1,\cdots,n$, we deduce
\[{\partial u\over\partial x_i }(x)=
\sum_{j=1}^n\int_{\Omega}{\partial^2 E\over\partial x_{i}
\partial y_j }(x,\cdot){f_j}\mathrm{dy}
+\int_{\Omega} {\partial E\over\partial x_i}(x,\cdot)
{f}\mathrm{dy}
+\int_{\Gamma} {\partial E\over\partial x_i}(x,\cdot)
{h}\mathrm{ds_y}.
\] 

Let ${\bf w}\in {\bf L}^{p'}(\Omega)$ be arbitrary such that $\|{\bf w}\|_{p',\Omega}=1$,
 applying the Fubini-Tonelli Theorem and next the H\"older inequality, it follows
\begin{eqnarray} \label{nuw}
\qquad \int_\Omega\nabla u\cdot{\bf w}\mathrm{dx}\leq \|{\bf f}
\|_{p,\Omega}
\left(\int_{\Omega}\left|
\sum_{i,j=1}^n\int_\Omega  {\partial^2 E\over \partial x_i \partial y_j}
(\cdot,y) w_i \mathrm{dx}\right|^{p'}\mathrm{dy}
\right)^{1/p'}\\
+\|f\|_{p,\Omega}
\left(\int_{\Omega}\left|\sum_{i=1}^n\int_\Omega {\partial 
E\over \partial x_i}(\cdot,y)w_i\mathrm{ dx}\right|^{p'}\mathrm{dy}
\right)^{1/p'}
\nonumber\\
+\|h\|_{p,\Gamma}
\left(\int_{\Gamma}\left|\sum_{i=1}^n
\int_\Omega {\partial E\over \partial x_i}(\cdot,y)w_i\mathrm{ dx}\right|^{p'}\mathrm{ds_y}\right)^{1/p'}.\nonumber
\end{eqnarray}

Let us estimate the last integral on RHS in (\ref{nuw}),
since the two others integrals are similarly bounded,
\begin{eqnarray*}
\mathcal{I}:=\left(\int_{\Gamma}\left|\int_\Omega\sum_{i=1}^n
 {\partial E\over \partial x_i}(x,y)w_i(x)\mathrm{
 dx}\right|^{p'}\mathrm{ds_y}\right)^{1/p'}
\leq \\
\leq |\Omega|^{1/p}  \left(
\int_\Omega\left(\int_{\Gamma}
|\nabla_x E(x,y)|^{p'}\mathrm{ds_y}\right)|{\bf w}(x)|^{p'}\mathrm{dx}
\right)^{1/p'},
\end{eqnarray*}
where $|\Omega|^{1/p}$ is due to the embedding $L^{p'}(\Omega)\hookrightarrow L^1(\Omega)$. 
 Considering that $\partial_{x_i} E(x,\cdot)\in W^{1,p'}(\Omega)$
uniformly for $x\in\Omega$ (cf. Proposition \ref{ddg}), consequently
 also $\nabla_x E(x,\cdot)\in {\bf W}^{1,p'}(\Omega)\hookrightarrow {\bf L}
^{p'(n-1)/(n-p')}(\Gamma)
\hookrightarrow {\bf L}^{p'}(\Gamma)$
 uniformly in $x\in\Omega$,
then we obtain
\[
\mathcal{I}\leq  |\Omega|^{1/p} |\Gamma|^{1/[p(n-1)]}K_{p'}\left({
C_1(\Omega,n,p') \over a_\#}+C_2 ( n,p',1/a_\#)\right).
\]

Finally,  inserting the  above inequality into
 (\ref{nuw}),  the proof of Proposition \ref{ppv} is finished. 
\end{proof}


\begin{thebibliography}{9}

\bibitem{adn}
S. Agmon, A. Douglis and  L. Nirenberg,
Estimates near the boundary for solutions of elliptic partial differential equations 
satisfying general boundary conditions I,
{\em Comm. Pure Appl. Math.} {12} (1959), 623-727.


\bibitem{agranovich}
 M.S. Agranovich,
Regularity of variational solutions to linear boundary value problems
in Lipschitz domains,
{\em Functional Analysis and Its Applications} {\bf 40} :4 (2006),  313-329,
Translated from Funktsional'nyi Analiz i Ego Prilozheniya, Vol. 40, No. 4 (2006), 83-103.
  
\bibitem{bv72}
 H. Beir\~ao da Veiga,
  Sur la regularit\'e des solutions de l'\'equation
${\rm div}~A(x,u,\nabla u)= B(x,u,\nabla u)$ avec des conditions
aux limites unilaterales et mel\'ees, {\it Ann. Mat. Pura Appl.}
{\bf 93} (1972), 173-230.

\bibitem{bg}
 L. Boccardo and  T. Gallou\"et,
 Non-linear elliptic and parabolic equations involving measure,
 {\em Journal of Functional Analysis} {\bf 87} (1989), 149-169.

\bibitem{bond}
J.F. Bonder and N. Saintier,
 Estimates for the Sobolev trace constant with critical exponent and applications,
{\em  Ann. Mat. Pura Appl.} {\bf 187} :4 (2008), 683-704.

\bibitem{bw2004}
S.-S. Byun and L. Wang, 
Elliptic equations with BMO coefficients in Reifenberg domains,
{\em Comm. Pure Appl. Math.}
 {\bf 57} :10 (2004), 1283-1310.

\bibitem{bw2005}
S.-S. Byun and L. Wang,
 The conormal derivative problem for elliptic equations with BMO coefficients 
on Reifenberg flat domains, 
{\em Proc. London Math. Soc.} {\bf 90} :1 (2005), 245-272.

\bibitem{zamm}
{L. Consiglieri}, The Joule-Thomson effect on the thermoelectric conductors,
{\em Z. Angew. Math. Mech.} {\bf 89} :3 (2009), 218-236. 

\bibitem{aduf}
{L. Consiglieri}, A limit model for thermoelectric equations,
{\em Annali dell'Universit\`a di Ferrara} {\bf 57} :2 (2011), 229-244.

\bibitem{lap}
{L. Consiglieri}, 
{\em Mathematical analysis of selected problems from fluid thermomechanics.
The  $(p-q)$ coupled fluid-energy systems}, 
Lambert Academic Publishing, Saarbr\"ucken, 2011.

\bibitem{epjp}
{L. Consiglieri}, 
On the posedness of thermoelectrochemical coupled systems,
{\em Eur. Phys. J. Plus} {\bf  128} :5  (2013), Article 47.

\bibitem{dahl}
B.E.J. Dahlberg,
Real Analysis and Potential Theory, {\em
Proceedings of the International Congress of Mathematicians
August 16-24, Warszawa}  (1983), 953-959.

\bibitem{aa}
 A. Dall'Aglio,
 Approximated solutions of equations with $L^1$ data.
Application to the $H$-convergence of quasi-linear parabolic equations,
 {\em Ann. Mat. Pura Appl.} {\bf 170} (1996), 207-240.

\bibitem{dauge}
 M. Dauge, 
{\em  Elliptic boundary value problems on corner domains: smoothness and 
asymptotics of solutions},
 Springer-Verlag, L.N. in Math. 1341,  Berlin 1988.

\bibitem{dif}
G. Di Fazio, 
$L^p$-estimates for divergence form elliptic equations with discontinuous coefficients, 
{\em Boll. Unione Mat. Ital. A}  {\bf 10} :2  (1996), 409-420.

\bibitem{dong}
H. Dong and D. Kim,
 Elliptic equations in divergence form with partially BMO coefficients,
{\em Arch. Rat. Mech. Anal.} {\bf 196} :1 (2010), 25-70.

\bibitem{elschk}
 J. Elschner, H.-C. Kaiser, J. Rehberg and G. Schmidt, 
$W^{1,q}$ regularity results for elliptic
transmission problems on heterogeneous polyhedra,
{\em Math. Mod. Meth. Appl. Sci.} {\bf 17} (2007), 593-615.

\bibitem{efv}
L. Escauriaza, E.B. Fabes and G. Verchota,
 On a regularity theorem for weak solutions to transmission problems with 
internal Lipschitz boundaries,
{\em Proc. Amer. Math. Soc.} {\bf 115} :4 (1992), 1069-1076.

\bibitem{geng}
J. Geng,
$W^{1,p}$ estimates for elliptic problems with Neumann boundary conditions
 in Lipschitz domains,
{\em Advances in Mathematics} {\bf 229} (2012), 2427-2448.

\bibitem{gia93}
 M. Giaquinta,
 {\em Introduction to Regularity Theory for Nonlinear Elliptic  Systems},
Birkhauser Verlag, Basel-Boston-Berlin 1993.

\bibitem{gt}
{D. Gilbarg} and {N. Trudinger},
{\em Elliptic partial differential equations of second order}, Springer-Verlag,
New York 1983.

\bibitem{grisv}
 P. Grisvard,
 {\em Elliptic problems in nonsmooth domains}, Monographs and
studies in mathematics {\bf 24}, Pitman, Boston-London 1985.

\bibitem{gro89}
 K. Gr\"oger,
 A $W^{1,p} -$ estimate for solutions to mixed boundary value problems
  for second order elliptic differential equations,
{\em Mathematische Annalen} {\bf 283} (1989), 679-687.

\bibitem{groreh}
 K. Gr\"oger and  J. Rehberg,
 Resolvent estimates in $W^{1,-p}$ 
  for second order elliptic differential operators in case of mixed boundary conditions.
{\em Mathematische Annalen} {\bf 285} (1989), 105-113.

\bibitem{gw}
M. Gr\"uter and K.-O. Widman, 
The Green function for uniformly elliptic equations,
{\em Manuscripta  Math.} {\bf 37} (1982), 303-342.

\bibitem{haller}
R. Haller-Dintelmann, H.C. Kaiser, and J. Rehberg, Elliptic model problems including mixed
boundary conditions and material heterogeneities,
{\em J. Math. Pures Appl.} {\bf 89} (2008), 25–48.

\bibitem{kp}
C.E. Kenig and J. Pipher,
The Neumann problem for elliptic equations
with non-smooth coefficients,
{\em Invent. Math.} {\bf 113} (1993), 447-509.

\bibitem{lu}
 O.A. Ladyzenskaya and  N.N. Ural'ceva, 
{\em Linear and quasilinear elliptic equations},
Mathematics in Science and Engineering {\bf 46}, 
Academic Press, New York-London 1968.

\bibitem{lv}
 Y.Y. Li and M. Vogelius,
Gradient estimates for solutions to divergence form elliptic equations
with discontinuous coefficients,
{\em Arch. Rat. Mech. Anal.} {\bf 53} (2000), 91-151.


\bibitem{lieb}
 E.H. Lieb, Sharp constants in the Hardy-Littlewood-Sobolev and related inequalities,
{\em  Ann. of Math.} {\bf  118} (1983), 349-374.

\bibitem{lsw}
W. Littman, G. Stampacchia and H.F. Weinberger,
Regular points for elliptic equations with discontinuous coefficients,
{\em Annali della Scuola Normale Superiore di Pisa,
 Classe di Scienze 3 s\'erie} {\bf 17} :1-2 (1963), 43-77.

\bibitem{mey}
  N.G. Meyers,
An $L^p$-estimates for the gradient of solutions of
second order elliptic divergence equations,
{\em Ann. Sci. Num. Sup. Pisa} {\bf 17} (1963), 189-206.

\bibitem{Ni} 
 L. Nirenberg,
On elliptic partial differential equations,
{\em Annali della Scuola Normale Superiore di Pisa} {\bf 13} (1959), 115--162.

\bibitem{ott}
K.A. Ott and R.M. Brown,
The mixed problem for the Laplacian in Lipschitz domains,
{\em Potential Analysis} {\bf 38} :4 (2013), 1333-1364.
 
\bibitem{p}
 A. Prignet, Conditions aux limites non homog\`enes
pour des probl\`emes elliptiques avec second membre mesure,
{\em  Ann. Fac. Sci. Toulouse Math.} \textbf{6} (1997), 297-318.  

\bibitem{rag99}
M.A. Ragusa, Regularity of solutions of divergence form elliptic equations, 
{\em
Proc. Amer. Math. Soc.} {\bf 128} (1999), 533-540.

\bibitem{sara}
D. Sarason,
Functions of vanishing mean oscillation,
{\em Trans. Amer. Math. Soc.} {\bf 207} (1975),
391-405.

\bibitem{stamp63-64}
G. Stampacchia,
\'Equations elliptiques du second ordre \`a coefficients discontinus,
{\em S\'eminaire Jean Leray} (1963-1964),  1-77.

\bibitem{sav}
G. Savar\'e, Regularity results for elliptic equations  in Lipschitz domains,
 {\em Journal of Functional Analysis} {\bf 152} (1998), 176-201.

\bibitem{tale}
G. Talenti,
Best constant in Sobolev inequality,
{\em Ann. Mat. Pura Appl.} {\bf 110} (1976), 353-372.

\end{thebibliography}
\end{document}